\documentclass[12pt]{amsart}

\usepackage{wasysym}
\usepackage{amsmath}
\usepackage{amssymb}
\usepackage{amsthm}
\usepackage{tikz}
\usetikzlibrary{matrix,arrows,backgrounds,shapes.misc,shapes.geometric,patterns,calc,positioning}
\usetikzlibrary{calc,shapes}
\usepackage{wrapfig}
\usepackage{epsfig}  
\usepackage[margin=1in]{geometry}
\usepackage{color}	 %
\input{xy}
\xyoption{poly}
\xyoption{2cell}
\xyoption{all}

\usepackage{lscape}

\usetikzlibrary{calc,shapes}
\usetikzlibrary{matrix,arrows,backgrounds,shapes.misc,shapes.geometric,patterns,calc,positioning}
\usetikzlibrary{calc,shapes}

\newcommand{\calf}{\mathcal{F}}
\newcommand{\cala}{\mathcal{A}}

\newcommand{\caln}{\mathcal{N}}

\newcommand{\calg}{\mathcal{G}}
\newcommand{\calh}{\mathcal{H}}

\newcommand{\calgSW}{{}_{SW}\calg}
\newcommand{\calgNE}{\calg^{N\!E}}

\makeatletter
\def\Ddots{\mathinner{\mkern1mu\raise\p@
\vbox{\kern7\p@\hbox{.}}\mkern2mu
\raise4\p@\hbox{.}\mkern2mu\raise7\p@\hbox{.}\mkern1mu}}
\makeatother

\newtheorem{thm}{Theorem}[section]
\newtheorem{prop}[thm]{Proposition}
\newtheorem{lem}[thm]{Lemma}
\newtheorem{cor}[thm]{Corollary}

\newtheorem{thmIntro}{Theorem}

\newtheorem*{cora}{Corollary}

\theoremstyle{defn}
\newtheorem{definition}[thm]{Definition}
\newtheorem{example}[thm]{Example}

\theoremstyle{remark}

\newtheorem{remark}[thm]{Remark}

\numberwithin{equation}{section}

\DeclareMathSizes{5}{3}{2}{2}

\newcommand{\gi}[3]{G_{#1}, G_{#2}, \ldots, G_{#3}}

\newcommand{\cfq}{[a_1,a_2,\ldots,a_n]}
\newcommand{\cfa}{[a_1,a_2,\ldots,a_n]}
\newcommand{\cfl}{[L_1,L_2,\ldots,L_n]}
\newcommand{\cfg}{\calg [a_1,a_2,\ldots,a_n]}

\newcommand\restr[2]{{
  \left.\kern-\nulldelimiterspace 
  #1 
  \vphantom{\big|} 
  \right|_{#2} 
  }}

  \def\sgn{\operatorname{sgn}}

\newcommand{\za}{\alpha}
\newcommand{\zb}{\beta}

\newcommand{\zD}{\Delta}
\newcommand{\ze}{\epsilon}
\newcommand{\zg}{\gamma}

\newcommand{\zL}{\Lambda}

\newcommand{\FF}{\mathsf{F}}
\newcommand{\GG}{\mathsf{G}}

\DeclareMathOperator{\Match}{Match}
\DeclareMathOperator{\cross}{cross}

\newcommand{\Int}{\textup{Int }}

\newcommand{\x}{\mathbf{x}}

\DeclareMathOperator{\match}{Match \, }

\begin{document}
\title{Cluster algebras and continued fractions}
\author{Ilke Canakci}
\address{Department of Mathematical Sciences,
Durham University,
Lower Mountjoy,
Stockton Road, Durham DH1 3LE,
UK}
\email{ilke.canakci@durham.ac.uk}
\author{Ralf Schiffler}\thanks{The first author  was supported by EPSRC grant number   EP/N005457/1, UK,  and  Durham University, and the second author was supported by the NSF-CAREER grant  DMS-1254567, and by the University of Connecticut.}
\address{Department of Mathematics, University of Connecticut, 
Storrs, CT 06269-3009, USA}
\email{schiffler@math.uconn.edu}




%
%
%
\begin{abstract}
We establish a combinatorial realization of continued fractions as quotients of cardinalities of sets. These sets are sets of perfect matchings of certain graphs, the snake graphs, that appear naturally in the theory of cluster algebras. To a continued fraction $\cfa$, we associate a snake graph $\calg\cfa$ such that the continued fraction is the quotient of the number of perfect matchings of $\calg\cfa$ and $\calg[a_2,\ldots,a_n]$. We also show that snake graphs are in bijection with continued fractions. 

We then apply this connection between cluster algebras and continued fractions in two directions. First, we use results from snake graph calculus to obtain new identities for the continuants of continued fractions. Then, we apply the machinery of continued fractions to cluster algebras and obtain explicit direct formulas for quotients of elements of the cluster algebra as continued fractions of Laurent polynomials in the initial variables. Building on this formula, and using classical methods for infinite periodic continued fractions, we also study the asymptotic behavior of quotients of elements of the cluster algebra. 
\end{abstract}

 \maketitle

\setcounter{tocdepth}{1}
\tableofcontents


\section{Introduction}
Cluster algebras were introduced by Fomin and Zelevinsky in 2002 in \cite{FZ1}. Originally motivated by the study of canonical bases in Lie theory,  the theory has gained  a tremendous development over the last 15 years, and cluster algebras are now connected to various areas of mathematics and physics, including Lie theory, quiver representations, combinatorics, dynamical systems, algebraic geometry, Teichm\"uller theory, and string theory.
 
 In this paper, we add another item to this list by providing a connection between cluster algebras and continued fractions. The method of continued fractions is a classical tool that is of fundamental importance in analysis and number theory, and also appears in many other areas of mathematics, for example in knot theory and hyperbolic geometry.
  
  The connection we are proposing is of combinatorial nature. The key combinatorial objects are the so-called snake graphs. These graphs appeared naturally in the theory of cluster algebras of surface type in \cite{Propp,MS,MSW,MSW2} and were studied in a systematic way in our previous work \cite{CS,CS2,CS3}. Each cluster variable in a cluster algebra of surface type is given by a combinatorial formula whose terms are parametrized by the perfect matchings of a snake graph \cite{MSW}.
 
 In the present paper, we construct a snake graph $\calg\cfa$ for every continued fraction $\cfa$ with $a_i\in\mathbb{Z}_{\ge 1}$. Our first main result is the following combinatorial interpretation of continued fractions as quotients of cardinalities of sets.

\begin{thmIntro}  If $m(\calg)$ denotes the number of perfect matchings of the graph $\calg$ then
 \[ \cfa=\frac{m(\calg\cfa)}{m(\calg[a_2,a_3,\ldots,a_n])}\]
 and the fraction on the right hand side  is reduced.
\end{thmIntro}
 
 In particular, this theorem implies that the number of terms in the Laurent expansion of cluster variables in cluster algebras of surface type is equal to the numerator of the associated continued fraction.

We then show that our construction actually provides a bijection between snake graphs and finite continued fractions in Theorem~\ref{thm bijections}, as well as a bijection between infinite snake graphs and infinite continued fractions in Theorem~\ref{thm bijectionsinf}. A direct consequence  is the following corollary.

\begin{cora}
 The number of snake graphs that have precisely $N$ perfect matchings is equal to $\phi(N)$, where $\phi$ is Euler's totient function.
\end{cora}
 
 This relation between snake graphs and continued fractions is useful in both directions. On the one hand, we can use it to interpret results on snake graphs from \cite{CS} as identities in terms of continued fractions. We present several examples in Theorem~\ref{graftingcf}, where we recover the classical identities for the convergents of continued fractions and  Euler's identities for continuants, but we also produce another family of identities  that we did not find in the literature on continued fractions.
 
 On the other hand, we can use the machinery of continued fractions in the study of cluster algebras. Let $\cala$ be a cluster algebra of surface type and initial cluster $(x_1,x_2,\ldots,x_N)$. Then the cluster variables $x_\zg$ of $\cala$ are in bijection with the arcs $\zg$ in the surface \cite{FST}, and can be computed via the snake graph formula of \cite{MSW} mentioned above. The same formula also defines an element $x_\zg$ for every \emph{generalized} arc,  where  generalized arcs  differ from arcs by allowing to have self-crossings.
 
 Now let $\zg$ be a generalized arc with snake graph $\calg_\zg=\calg\cfa$. Then there exists a unique generalized arc $\zg'$ whose snake graph $\calg_{\zg'}$ is the subgraph $\calg[a_2,a_3,\ldots,a_n]$ of $\calg_\zg$. 
 For each $a_i$, we construct a Laurent polynomial $L_i$ in the initial cluster variables $x_1,x_2,\ldots,x_N$ depending on $\zg$, and we prove the following result.
 
\begin{thmIntro}\label{thm 1.2}
 For every generalized arc $\zg$, we have the following identity in the field of fractions of the cluster algebra $\cala$
 \[
\frac{x_\zg}{x_{\zg'}}=\cfl.\]
Moreover, $x_\zg$ and $x_{\zg'}$ are relatively prime in the ring of Laurent polynomials and thus the left hand side of the equation is reduced.

In particular, if $\zg$ and $\zg'$ have no self-crossing then the left hand side of the equation is a quotient of cluster variables.
\end{thmIntro}
 
 This formula is of considerable computational interest, since the $L_i$ are given by an explicit formula, see Proposition \ref{prop hi}. 
 
 Theorem \ref{thm 1.2} suggests that one should study the quotients of certain elements of the cluster algebra and their asymptotic behavior. Using  techniques from continued fractions, we obtain the following result. In this theorem, we work in the cluster algebra of the once punctured torus for convenience.

\begin{thmIntro}\label{thm 1.3}
Let $\cala$ be 
 the cluster algebra of the once punctured torus with initial cluster $\x=(x_1,x_2,x_3)$. 
 Let $x'_1=(x_2^2+x_3^2)/x_1$ denote the cluster variable obtained from $\x$ by the single mutation $\mu_1$. Let  $u(i)$ denote the cluster variable obtained from $\x$ by the mutation sequence  $\mu_1\mu_2\mu_1\mu_2\mu_1\cdots$ of length $i$  and $\calg$ be its snake graph.
Let  $v(i)$ be the  Laurent polynomial corresponding to the snake graph 
obtained from $\calg$ by removing the south  edge $e_0$.  \\
(a) The quotient $u(i)/v(i)$   converges as $i$ tends to infinity and its limit is  
\[\za =  \frac{{ x'_1-x_1} +\sqrt{\left({ x'_1-x_1}\right)^2+4x_3^2}}{2x_3}. \]

The quotient of the cluster variables $u(i)/u(i-1)$ converges as $i$ tends to infinity and its limit is 
\[\zb
= \frac{x_1'+x_1+\sqrt{(x_1'-x_1)^2+4x_3^2}}{2x_2} .\]
(b)  Let $z={ \frac{ x'_1-x_1}{x_3} }$. Then $\za$ has the following  periodic continued fraction expansion
\[\za= \left\{\begin{array}{ll}\left[\,\overline{ z}\,\right] , &\textup{if $z  >0$;}\\ \\
\left[ 0, \overline{-z }\,\right] &\textup{if $z < 0$} .\end{array}\right.\]
 \end{thmIntro}
Similar formulas hold for other mutation sequences. In particular the analogous results  for the annulus with two marked points
is obtained  by setting $x_3=1$.

\smallskip
The results of this paper have applications in several lines of  research which will be studied in future work, including  cluster algebras and cluster categories of infinite rank \cite{HJ,GG, LP, CF}, continuous cluster categories \cite{IT},  Euler characteristic of quiver Grassmannians (see Corollary \ref{corgr} for a first result), string modules and string objects in the module categories and cluster categories of triangulated marked surfaces \cite{BZ,QZ, Labardini}, combinatorial structure of continued fractions, integer sequences, and friezes \cite{Cox}.

Continued fractions are ubiquitous in mathematics and  have also appeared in the context of cluster algebras before, although unrelated to our results.  
Nakanishi and Stella used them for constructing an initial seed of a cluster algebra in order to prove periodicity in sine-Gordon $Y$-systems in \cite{NaSt}. Di Francesco and Kedem obtained a very particular kind of (generalized) continued fractions in their computation of generating functions of cluster variables that lie along a specific sequence of mutations in a cluster algebra associated to a Cartan matrix. Their (generalized) continued fractions are of the form 
$a_1+\cfrac{b_1}{a_2+\cfrac{b_2}{a_3\cdots}}$ with numerators $b_i\ne 1$, and the $a_i$ depending on the $b_i$. When specializing all $b_i=1$, one obtains continued fractions of the form $[1,1,2,2,\ldots,2]$ \cite{DFK1}.
Our two approaches have an interesting overlap in the case of the cluster algebra of the Kronecker quiver which arises as a special case in two different ways. For Di Francesco and Kedem it corresponds to the case of the Cartan matrix $(a_{11})=(2)$ of rank 1, and for us to the annulus with 2 marked points on the boundary. In section  \ref{sect 5}, we compute examples of the asymptotic behavior of cluster variables for the torus with one puncture, from which the Kronecker case can be obtained by specializing one of the variables to 1. It would be interesting to compare our limit function with the generating function in \cite{DFK1} in this case.
 
The paper is organized as follows. In section \ref{sect prel}, we recall definitions and facts about snake graphs and continued fractions.  The construction of the snake graph $\calg\cfa$ of a continued fraction $\cfa$ is given in section \ref{sect 1}, and the bijection between snake graphs and continued fractions in section \ref{sect bijections}. In section \ref{sect 4} we list several identities of continued fractions that follow immediately from the snake graph calculus of \cite{CS}. The relation to cluster algebras is given in sections \ref{sect cluster}  and \ref{sect 5}.  In particular, the proof of Theorem \ref{thm 1.2} is  given in section \ref{sect cluster} and  Theorem~\ref{thm 1.3} in section \ref{sect 5}.
\section{Preliminaries}\label{sect prel}

\subsection{Cluster algebras from surfaces}\label{sect 1.1}
Cluster algebras were introduced by Fomin and Zelevinsky in \cite{FZ1}. 
A cluster algebra is a $\mathbb{Z}$-subalgebra of a field of rational functions $\mathbb{Q}(x_1,\ldots,x_{N})$
in $N$ variables. The cluster algebra is given by a set of generators, the cluster variables, which is constructed 
recursively from the initial cluster variables $x_1.\ldots,x_N$ by a process called mutation. It is known that each cluster variable is a Laurent polynomial in $x_1,\ldots,x_N$ with integer coefficients \cite{FZ1} and that these coefficients are non-negative \cite{LS4}. We refer the reader to \cite{FZ4} for the precise definition of cluster algebras.

In this paper, we are mainly interested in the special type of cluster algebras which are associated to  marked surfaces, see \cite{FST}.
Let $S$ be a connected oriented 2-dimensional Riemann surface with
(possibly empty)
boundary.  Fix a nonempty set $M$ of {\it marked points} in the closure of
$S$ with at least one marked point on each boundary component. The
pair $(S,M)$ is called a \emph{bordered surface with marked points}. Marked
points in the interior of $S$ are called \emph{punctures}.  

An \emph{arc} $\zg$ in $(S,M)$ is a curve in $S$, considered up
to isotopy, such that 
\begin{itemize}
\item[(a)] the endpoints of $\zg$ are in $M$;
\item[(b)] $\zg$ does not cross itself, except that its endpoints may coincide;
\item[(c)] except for the endpoints, $\zg$ is disjoint from $M$ and
  from the boundary of $S$,
\item[(d)] $\zg$ does not cut out an unpunctured monogon or an unpunctured bigon. 
\end{itemize}  
A \emph{generalized arc} is a curve which satisfies conditions (a),(c) and (d), but it can have selfcrossings.
Curves that connect two
marked points and lie entirely on the boundary of $S$ without passing
through a third marked point are  called \emph{boundary segments}.
By (c), boundary segments are not  arcs.

For any two arcs $\zg,\zg'$ in $(S,M)$, let $e(\zg,\zg')$ be the minimal
number of crossings of 
arcs $\za$ and $\za'$, where $\za$ 
and $\za'$ range over all arcs isotopic to 
$\zg$ and $\zg'$, respectively.
We say that two arcs $\zg$ and $\zg'$ are  \emph{compatible} if $e(\zg,\zg')=0$. 

An \emph{ideal triangulation} is a maximal collection of
pairwise compatible arcs (together with all boundary segments). 
The arcs of a 
triangulation cut the surface into \emph{ideal triangles}. 

In \cite{FST}, the authors associated a cluster algebra $\cala(S,M)$ to 
any bordered surface with marked points $(S,M)$, and showed that
the clusters of  $\cala(S,M)$ are in bijection 
with triangulations  of $(S,M)$, and
the cluster variables of  $\cala(S,M)$  are  in bijection
with the (tagged) arcs of $(S,M)$.

\subsection{Snake graphs}\label{sect 1.2} 
 Abstract snake graphs   have been introduced and studied \cite{CS,CS2,CS3} motivated by the snake graphs  appearing in the combinatorial formulas for elements in cluster algebras of surface type  in \cite{Propp,MS,MSW,MSW2}. In this section, we recall the main definitions.
Throughout
we fix the standard orthonormal basis of the plane.

 A {\em tile} $G$ is a square in the plane whose sides are parallel or orthogonal  to the elements in    the  fixed basis. All tiles considered will have the same side length.
\begin{center}
  {\scriptsize
\begingroup%
  \makeatletter%
  \providecommand\color[2][]{%
    \errmessage{(Inkscape) Color is used for the text in Inkscape, but the package 'color.sty' is not loaded}%
    \renewcommand\color[2][]{}%
  }%
  \providecommand\transparent[1]{%
    \errmessage{(Inkscape) Transparency is used (non-zero) for the text in Inkscape, but the package 'transparent.sty' is not loaded}%
    \renewcommand\transparent[1]{}%
  }%
  \providecommand\rotatebox[2]{#2}%
  \ifx\svgwidth\undefined%
    \setlength{\unitlength}{68.59006958bp}%
    \ifx\svgscale\undefined%
      \relax%
    \else%
      \setlength{\unitlength}{\unitlength * \real{\svgscale}}%
    \fi%
  \else%
    \setlength{\unitlength}{\svgwidth}%
  \fi%
  \global\let\svgwidth\undefined%
  \global\let\svgscale\undefined%
  \makeatother%
  \begin{picture}(1,0.70549664)%
    \put(0,0){\includegraphics[width=\unitlength]{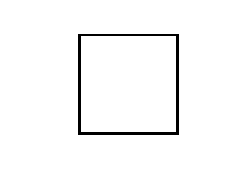}}%
    \put(0.49750979,0.31061177){\color[rgb]{0,0,0}\makebox(0,0)[lb]{\smash{$G$}}}%
    \put(0.00429408,0.31029854){\color[rgb]{0,0,0}\makebox(0,0)[lb]{\smash{West}}}%
    \put(0.79324631,0.31021311){\color[rgb]{0,0,0}\makebox(0,0)[lb]{\smash{East}}}%
    \put(0.36461283,0.61868105){\color[rgb]{0,0,0}\makebox(0,0)[lb]{\smash{North}}}%
    \put(0.36019916,0.00759722){\color[rgb]{0,0,0}\makebox(0,0)[lb]{\smash{South}}}%
  \end{picture}%
\endgroup%
}
\end{center}
We consider a tile $G$ as  a graph with four vertices and four edges in the obvious way. A {\em snake graph} $\calg$ is a connected planar graph consisting of a finite sequence of tiles $ \gi 12d$ with $d \geq 1,$ such that
$G_i$ and $G_{i+1}$ share exactly one edge $e_i$ and this edge is either the north edge of $G_i$ and the south edge of $G_{i+1}$ or the east edge of $G_i$ and the west edge of $G_{i+1}$,  for each $i=1,\dots,d-1$.
An example is given in Figure \ref{signfigure}.
\begin{figure}
\begin{center}
  {\tiny \scalebox{0.9}{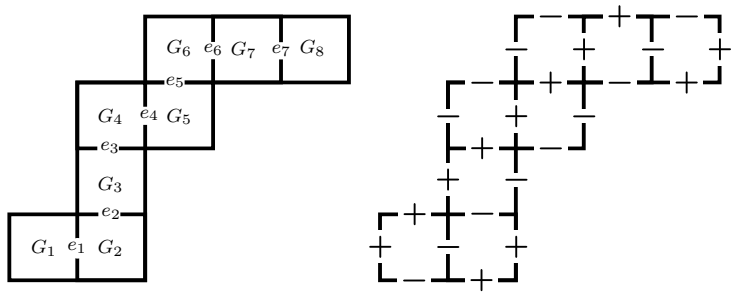}}
 \caption{A snake graph with 8 tiles and 7 interior edges (left);
 a sign function on the same snake graph (right)} 
 \label{signfigure}
\end{center}
\end{figure}
The graph consisting of two vertices and one edge joining them is also considered a snake graph.

The $d-1$ edges $e_1,e_2, \dots, e_{d-1}$ which are contained in two tiles are called {\em interior edges} of $\calg$ and the other edges are called {\em boundary edges.}  
Denote by $\Int(\calg)=\{e_1,e_2,\ldots,e_{d-1}\}$  the set of interior edges of $\calg$.
We will always use the natural ordering of the set of interior edges, so that $e_i$ is the edge shared by the tiles $G_i$ and $G_{i+1}$.

We denote by  $\calgSW$ the 2 element set containing the south and the west edge of the first tile of $\calg$ and by $\calgNE$ the 2 element set containing the north and the east edge of the last tile of $\calg$. We will sometimes refer to the two edges in $\calgSW$ as the \emph{south edge} of $\calg$ and the  \emph{west edge} of $\calg$, respectively. Similarly, the two edges in $\calgNE$ will be called the \emph{north edge} of $\calg$ and the  \emph{east edge} of $\calg$, respectively. 
 If $\calg$ is a single edge, we let $\calgSW=\emptyset$ and $\calgNE=\emptyset$. 

A snake graph $\calg$ is called {\em straight} if all its tiles lie in one column or one row, and a snake graph is called {\em zigzag} if no three consecutive tiles are straight.
 We say that two snake graphs are \emph{isomorphic} if they are isomorphic as graphs.

A {\em sign function} $f$ on a snake graph $\calg$ is a map $f$ from the set of edges of $\calg$ to $\{ +,- \}$ such that on every tile in $\calg$ the north and the west edge have the same sign, the south and the east edge have the same sign and the sign on the north edge is opposite to the sign on the south edge. See Figure \ref{signfigure} for an example.
%
%

Note that on every snake graph  there are exactly two sign functions. A snake graph is  determined up to symmetry by its sequence of tiles together with a sign function on its interior edges.
If $\calh$ is a subgraph of $\calg$, we write $\calg\setminus\calh$ for the full subgraph of $\calg$ whose vertices are in $\calg \setminus\calh$. For example, if $\calg$ is the snake graph of 8 tiles in Figure~\ref{signfigure} then $\calg\setminus G_8$ is the snake graph consisting of the first 6 tiles. 

\subsubsection{Infinite snake graphs} We define \emph{infinite snake graphs}  $\calg$ analogously using an infinite sequence of tiles $(G_1,G_2,\ldots)$. If $\calg$ is an infinite snake graph then $\calgSW$ is defined as above, but $\calgNE $ does not exist.

\subsection{Labeled snake graphs from surfaces} \label{sect label}
In this subsection we recall the definition of snake graphs of \cite{MSW}. We follow the exposition in \cite{S}.

Let $T$ be an ideal triangulation of a surface $(S,M)$ and let
$\zg$ be an   arc in $(S,M)$ which is not in $T $. 
Choose an orientation on $\zg$, let $s\in M$ be its starting point, and let $t\in M$ be its endpoint. 
 Denote by 
\[ s=p_0, p_1, p_2, \ldots, p_{d+1}=t
\]
the points of intersection of $\zg$ and $T $ in order.  For $j=1,2,\ldots,d$, 
let  $\tau_{i_j}$ be the arc of $T $ containing $p_j$, and let 
$\zD_{j-1}$ and 
$\zD_{j}$ be the two ideal triangles in $T $ 
on either side of 
$\tau_{i_j}$.  
Then, for $j=1,\ldots,d-1,$ the arcs  $\tau_{i_j}$ and $\tau_{i_{j+1}}$ form two sides of the  triangle $\zD_j$ in $T$ and we  define $e_j$ to be the third arc in this triangle, see Figure~\ref{gammafig}.
\begin{figure}
\begin{center}
  {\scriptsize \scalebox{1}{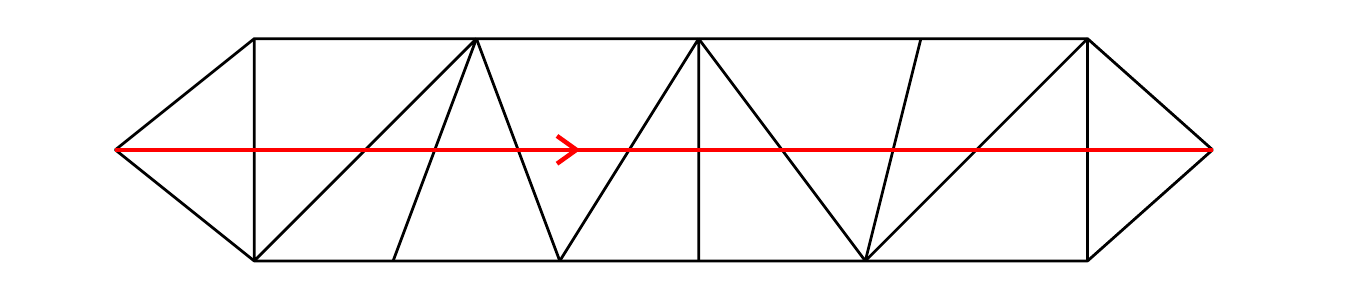}}
 \caption{The arc $\zg$ passing through $d+1$ triangles $\zD_0,\ldots,\zD_d$} 
 \label{gammafig}
\end{center}
\end{figure}

Let $G_j$ be the  quadrilateral in $T$ that contains $\tau_{i_j}$ as a diagonal. We will think of $G_j$ as a {\em tile} as in section \ref{sect 1.2}, but now the edges of the tile are arcs in $T$ and thus are labeled edges. We also think of the tile $G_j$ itself being labeled by the diagonal $\tau_{i_j}$.

Define a \emph{sign function} $f$ on the edges $e_1,\ldots,e_d$ by 
\[f(e_j)=\left\{\begin{array}{ll}+1 &\textup{if $e_j$ lies on the right of $\gamma$ when passing through $\zD_j$}\\-1 &\textup{otherwise.}\end{array}\right.\]

The  labeled snake graph  $\calg_\zg=(G_1,\ldots,G_d)$ with tiles $G_i$ and sign function $f$ is called the {\em snake graph associated to the arc} $\zg$. Each edge $e$ of $\calg_\zg$ is labeled by an arc $\tau(e)$ of the triangulation $T$. We define the \emph{weight} $x(e)$ of the edge $e$ to be cluster variable associated to the arc $\tau(e)$. Thus $x(e)=x_{\tau(e)}$.

The main result of \cite{MSW} is a combinatorial formula for the cluster variables in $\cala(S,M)$. Let $\zg$ be an arc and $x_\zg$ the corresponding cluster variable.  A \emph{perfect matching} $P$ of $\calg_\zg$ is a subset of the set of edges of $\calg_\zg$ such that each vertex of $\calg_\zg$ is incident to exactly one edge in $P$.
Then, if the triangulation has no self-folded triangles, we have
\[x_\zg=\frac{1}{\cross(\calg_\zg)} \sum_{P\in \Match \calg_\zg} x(P),\]
where the sum runs over all perfect matchings of $\calg_\zg$,  the summand $x(P)=\prod_{e\in P} x(e)$ is the weight of the perfect matching $P$, and $\cross(\calg)=\prod_{j=1}^d x_{i_j}$ is the product (with multiplicities) of all initial cluster variables whose arcs are crossed by $\zg$.
In the presence of self-folded triangles, the combinatorial formula in \cite{MSW} is slightly more complicated.

\subsection{Continued fractions}\label{sect 1.3}

Continued fractions are  remarkably influential in many areas of mathematics. For a standard introduction, we refer to  \cite[Chapter 10] {HW}. For a more extensive treatment of continued fractions see \cite{Perron}.

A \emph{finite continued fraction} is a function 
\[[a_1,a_2,\ldots,a_n]= a_1+\cfrac{1}{a_2+\cfrac{1}{a_3+\cfrac{1}{\ddots +\cfrac{1}{a_n}}}}\]
of $n$ variables $a_1,a_2,\ldots,a_n$. 
Similarly, an \emph{infinite continued fraction} is a function 
\[[a_1,a_2,\ldots]= a_1+\cfrac{1}{a_2+\cfrac{1}{a_3+\cfrac{1}{\ddots }}}\]
of infinitely many variables $a_1,a_2,\ldots$ 
We use the notation $[\overline{a}]$ for the periodic continued fraction $[\overline{a}]=[a,a,a,\ldots]$, where $a_i=a$, for all $i$.
Note that 
\[\cfa=a_1+\frac{1}{[a_2,\ldots,a_n]}\qquad \textup{ and }\qquad[\overline{a}]=a+\frac{1}{[\overline{a}]}.\]
 
In sections \ref{sect 1} -- \ref{sect 4} of this paper, we are interested in finite continued fractions where each $a_i$ is an integer. However, in section~\ref{sect cluster}, we consider finite continued fractions of Laurent polynomials  as well as finite continued fractions of complex numbers, and in section \ref{sect 5}, infinite continued fractions of Laurent polynomials. 

We say that a continued fraction $[a_1,a_2,\ldots,a_n]$ is  \emph{positive} if each $a_i\in \mathbb{Z}_{\ge 1}$.
We say that a continued fraction $[a_0,a_1,\ldots,a_n]$ is  \emph{simple} if $a_0\in\mathbb{Z} $ and $a_i\in\mathbb{Z}_{\ge 1}$, for each $i\ge 1$.

Note that if $\cfa$ is a simple continued fraction with $a_n=1$, then $\cfa=[a_1,a_2,\ldots,a_{n-1}+1]$. Thus it suffices to consider continued fractions whose last coefficient is at least 2. The following classical result shows that this  is the only ambiguity for finite continued fractions. 
\begin{prop}
 \cite[Theorem 162]{HW} \label{thm162}
 \begin{itemize}
\item [\textup{(a)}]
There is a bijection between  $\mathbb{Q}_{> 1}$  and the set of finite positive continued fractions whose last coefficient is at least 2. 
\item [\textup{(b)}]
There is a bijection between  $\mathbb{Q}$  and the set of finite simple continued fractions whose last coefficient is at least 2. 
\end{itemize}
\end{prop}

\begin{example}
We have $[2,3,4]=2+\cfrac{1}{3+\cfrac{1}{4}}=2+\cfrac{4}{13} =\cfrac{30}{13}$. \end{example}
We shall need the following general fact about continued fractions.

\begin{lem}
\label{lem cf} 
$[ra_1,r^{-1}a_2,ra_3,\ldots,r^{(-1)^{n+1}} { a_n}] =r\,\cfa$. 
\end{lem}
\begin{proof}
If $ n=1$, both sides of the equation are equal to $ra_1$. Suppose $n>1$. Then
\[\begin{array}
 {rcl} [ra_1,r^{-1}a_2,ra_3,\ldots,r^{(-1)^{n+1}}  a_n] &=& \displaystyle ra_1+\frac{1}{[r^{-1}a_2,ra_3,\ldots,r^{(-1)^{n+1}}  a_n] }\\ \\
 &=&\displaystyle ra_1+\frac{1}{r^{-1}[a_2,a_3,\ldots, a_n] }\quad \textup{(by induction)} \\ \\
   &=&\displaystyle r \left( a_1+\frac{1}{[a_2,a_3,\ldots, a_n]}\right)\\ \\
   &=&
  r\,
  \cfa. \hfill \qedhere
\end{array}\]
\end{proof}
\begin{example}
We have $[2,3,4]=\cfrac{30}{13}$, and $[2\cdot 13, \,\cfrac{3}{13}\,, 4\cdot13]=30$. 
\end{example}

For infinite continued fractions we have the following classical result.
\begin{prop}\cite[Theorem 170]{HW}
 \label{thm162R}
\begin{itemize}
\item [\textup{(a)}]
There is a bijection between  $\mathbb{R}_{>1}\setminus \mathbb{Q}_{> 1}$  and the set of infinite positive continued fractions.
\item [\textup{(b)}]
There is a bijection between  $\mathbb{R}\setminus\mathbb{Q}$  and the set of infinite simple continued fractions.
\end{itemize}
\end{prop}

\section{The snake graph of a continued fraction}\label{sect 1}
We have seen in section \ref{sect 1.2} that snake graphs appear naturally in the theory of cluster algebras from surfaces, where they are used to compute the Laurent expansions of the cluster variables. The terms in the Laurent polynomial of a cluster variable are parametrized by the perfect matchings of the associated snake graph. 

In this section, we present a new connection between snake graphs and continued fractions. 
For every positive continued fraction $\cfa$, we construct  a snake graph $\cfg$ in such a way that the number of perfect matchings of the snake graph is equal to the numerator of the continued fraction.

Recall that a snake graph $\calg$ is determined by a sequence of tiles $G_1,\ldots,G_d$ and a sign function $f$ on the interior edges $ \Int \calg$ of $\calg$. As usual, we denote the interior edges by $e_1,\ldots,e_{d-1}$.  Denote by $e_0\in\calgSW$ the south edge of $\calg$ and choose an edge $e_d\in\calgNE$. The sign function can be extended in a unique way to all edges of $\calg$. In particular, we obtain a sign sequence 
\begin{equation}
 \label{seq}
 (f(e_0) , f(e_1) ,\ldots ,f(e_{d-1}), f(e_d)).
\end{equation}
This sequence uniquely determines the snake graph and a choice of a northeast edge $e_d\in \calgNE$. 

\medskip

Now let $\cfq$ be a positive continued fraction, and let $d= a_1+a_2+\cdots +a_n -1$.
Consider the following sign sequence
\begin{equation}
 \label{eqsign} 
\begin{array}{cccccccc}
  ( \underbrace{ -\ze,\ldots,-\ze},&  \underbrace{ \ze,\ldots,\ze},&  \underbrace{ -\ze,\ldots,-\ze},& \ldots,&  \underbrace{\pm\ze,\ldots,\pm\ze}) ,  \\
 a_1 & a_2 & a_3&\ldots&a_n
\end{array} 
\end{equation}
where  $\ze \in \{ -,+ \}$, $-\ze=\left\{\begin{array}{cl} + &\textup{if $\ze=-$;}\\ - &\textup{if $\ze=+$.}\end{array}\right. $  We define  $\sgn(a_i) = \left\{\begin{array}{cl} -\ze &\textup{if $i$ is odd;}\\ \ze &\textup{if $i$ is even.}\end{array}\right. $
 Thus each integer $a_i$ corresponds to a maximal subsequence of constant sign $\sgn(a_i)$ in the sequence (\ref{eqsign}).

We let $\ell_i$ denote  the position of the last term in the $i$-th subsequence, thus  
$\ell_i = \sum_{j=1}^{i} a_{j}.$

\begin{definition}\label{def}
 The snake graph
 $\calg[a_1,a_2,\ldots,a_n] $ of the positive continued fraction $[a_1,a_2,\ldots,a_n]$ is the snake graph with $d$ tiles determined by the sign sequence (\ref{eqsign}). 
\end{definition}
 An example is given in Figure \ref{fig 8437}. In the special case where $n=1$ and $a_1=1$, this definition means that $\calg[1]$ is a single edge.

\begin{figure}
\begin{center}
 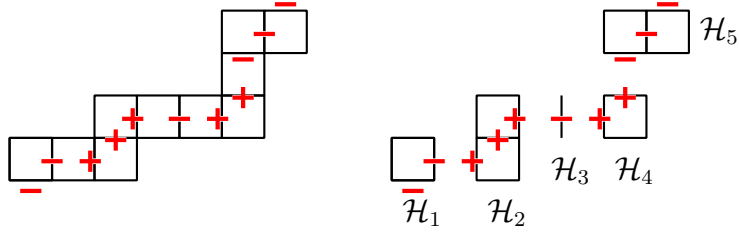
 \caption{The snake graph of the continued fraction $[2,3,1,2,3]$. The sign sequence $(-,-,+,+,+,-,+,+,-,-,-)=(f(e_0),f(e_1),\ldots, f(e_{10}))$ is given in red. 
 The sign changes occur in the tiles $G_{\ell_i}$, with $\ell_i=2,5,6,8$. The subgraphs $\calh_1,\ldots,\calh_5$ on the right are obtained by  removing the boundary edges of the tiles $G_{\ell_i},$  for $i=1,\ldots, n.$} \label{fig 8437}
\end{center}
\end{figure}

Let $f\colon\{e_0,e_1,\ldots,e_d\}\to\{\pm\}$ denote the sign function of $\cfg$ induced by the sign sequence.

\begin{lem}\label{lem1}
For each $i =1, 2, \ldots, n-1$, the subsnake graph of $\calg\cfq$ consisting of the three consecutive tiles $(G_{\ell_i-1},G_{\ell_i},G_{\ell_i+1}) $ is straight.
\end{lem}
 
\begin{proof} This follows immediately from the construction, since $f(e_{\ell-1 })\neq f(e_{ \ell})$.
\end{proof}

Next we define zigzag subsnake graphs $\calh_1,\ldots,\calh_n$ of $\calg=\calg\cfq $ as follows. Let
\[
\begin{array}
 {rcl}
 \calh_1 &=& (G_1,\ldots ,G_{\ell_1-1})  \\
\calh_{2} &=&(G_{\ell_1+1},\ldots ,G_{\ell_2-1})  \\ 
&\vdots&\\
\calh_{i} &=&(G_{\ell_{i-1}+1},\ldots ,G_{\ell_i-1})  \\ 
&\vdots&\\
 \calh_n &=&(G_{\ell_{n-1}+1},\ldots,G_d),
 \end{array}\]
where $(G_j,\ldots,G_k)$ is the subsnake graph of $\calg$ consisting of the tiles $G_j,G_{j+1},\ldots,G_k$, if $j\le k$, and  $(G_{j+1},\ldots,G_j)$ is the single edge $e_j$,  see Figure \ref{fig 8437} for an example.

\begin{lem}\label{lem2} For each $i =1, 2, \ldots, n$, the graph
 $\calh_i$ is a zigzag snake graph with $a_i-1$ tiles. 
  In particular $\calh_i\cong\calg[a_i]$  is the snake graph of the continued fraction $[a_i]$ with a single coefficient $a_i$. 
\end{lem}
\begin{proof}
 By construction, the interior edges of $\calh_i$  all have the same sign $\sgn(a_i)$, and $\calh_i$ has $a_i-1$ tiles.  This implies  that  $\calh_i\cong\calg[a_i]$ and that it is a zigzag snake graph. 
\end{proof}

We are now ready for the main result of this section.
\begin{thm}\label{thm1} \begin{itemize}
\item [\textup{(a)}]
 The number of perfect matchings of $\calg[a_1,a_2,\ldots,a_n] $ is equal to the numerator of the continued fraction $[a_1,a_2,\ldots,a_n]$. 
\item [\textup{(b)}]
 The number of perfect matchings of $ \calg[a_2,a_3,\ldots,a_n] $ is equal to the denominator of the continued fraction $[a_1,a_2,\ldots,a_n]$. 
 \item[\textup{(c)}] 
 If  $m(\calg)$ denotes the number of perfect matchings of $\calg$ then
 \[ \cfa =\frac{m(\calg\cfa)}{m(\calg[a_2,\ldots,a_n])}.\] 
 \end{itemize}
\end{thm}
\begin{proof} (a)
If $n=1$, then $\calg[a_1]$ is a zigzag snake graph with $a_1-1$ tiles.  If $a_1=1$, this graph is a single edge and it has exactly one perfect matching. If $a_1>1$, then there is exactly one perfect matching of $\calg[a_1]$ that contains the edge $e_0$. On the other hand, if a perfect matching does not contain $e_0$, then it must contain the other edge  $b_0$ in $\calgSW[a_1]$ and it must be a perfect matching on the remaining part which is a zigzag snake graph with $a_1-2$ tiles. By induction this gives a total of $a_1$ perfect matchings. 

Now suppose that $n>1$, and let $P$ be a perfect matching  of the snake graph $\calg[a_1,a_2,\ldots,a_n]$. We will denote by $\mathcal{N}\cfq$  the numerator of the continued fraction $\cfq$.  By Lemma \ref{lem1}, the tiles $G_{\ell_1-1},G_{\ell_1}$ and $G_{\ell_1+1}$ form a straight subsnake graph of $\calg\cfq$. 
If $P$ does not contain the two boundary edges of $G_{\ell_1}$ then the restrictions of $P$ to $\calh_1\cong\calg[a_1]$ and $\calg\setminus\calh_1\cong\calg[a_2,a_3,\ldots,a_n]$ are perfect matchings of these subgraphs. By induction this gives $a_1+ \mathcal{N} [a_2,a_3,\ldots,a_n]$  perfect matchings. On the other hand, if $P$ contains the boundary edges of $G_{\ell_1}$ then, since $\calh_1$ and $\calh_2$ are zigzag snake graphs, the restriction of $P$ to $\calh_1$ and $\calh_2$ consists of boundary edges only, and there is a unique such matching on  $\calh_1$ and on $\calh_2$. Moreover, the restriction of $P$ to $\calg[a_3,a_4,\ldots,a_n]$ is a perfect matching of that subgraph. By induction, this gives $ \mathcal{N}[a_3,a_4,\ldots,a_n]$  perfect matchings. Putting these two cases together, we get a total of 
\begin{equation}\label{graftingeq}
a_1\,\mathcal{N}[a_2,a_3,\ldots,a_n] +\mathcal{N} [a_3,a_4,\ldots,a_n] \end{equation}
perfect matchings. 

For the remainder of this proof, denote by $N$ and $D$ the numerator and the denominator of the continued fraction $[a_3,\ldots,a_n]$. Thus $[a_3,\ldots,a_n]=\frac{N}{D}$ and $N,D$ are relatively prime. Then
\[ \cfa \ =\ a_1+\cfrac{1}{a_2+\cfrac{D}{N}} \ =\ a_1+\cfrac{N}{a_2N+D} \ =\ \cfrac{a_1( a_2N+D)+N}{a_2N+D}.
\]
Note that this fraction is reduced, because $N$ and $D$ are relatively prime. Thus
\begin{equation}\label{eqDN}
 \caln \cfa \ =\ a_1( a_2N+D)+N.
\end{equation}
Furthermore, from the identity
\[ [a_2,\ldots,a_n]=a_2+\frac{D}{N}=\frac{a_2N+D}{N}\] 
we see that $\caln[a_2,\ldots,a_n]=a_2N+D$, and hence (\ref{eqDN})
implies that
\[\caln\cfa=a_1\caln[a_2,\ldots,a_n]+\caln[a_3,\ldots,a_n].\]
This completes the proof of (a).

Statement (b) follows directly from (a) because the denominator of $\cfq$ is equal to the numerator of $[a_2,\ldots,a_n]$, and (c) is obtained by combining (a) and (b).
%
\end{proof}

\begin{example}\label{exfib}
 For $\cfa=[1,1,\ldots,1]$, the  snake graph $\calg=\calg[1,1,\ldots, 1] $ is the straight snake graph with $ n-1$ tiles and its   number of perfect matchings is the $ n+1$-st Fibonacci number. The first few values are given in the table below.
\begin{center} 
 \begin{tabular}{c||c|c|c|c|c|c|c|c|c|c}
 $ n$ & 1 & 2&3&4&5&6&7&8&9&10\\ \hline
$ m(\calg)$&1&2&3&5&8&13&21&34&55&89
\end{tabular}
\end{center}

\end{example}

\begin{example}\label{expell}
 For $\cfa=[2,2,\ldots,2]$, the  snake graph $\calg=\calg[2,2,\ldots, 2] $ has a staircase shape where each horizontal and each vertical segment has exactly 3 tiles, and in total there are $n-1$ segments. The number of perfect matchings  of $\calg$ is the $ n+1$-st Pell number. The first few values are given in the table below.
\begin{center} 
 \begin{tabular}{c||c|c|c|c|c|c|c|c|c|c}
 $ n$ & 1 & 2&3&4&5&6&7&8&9&10\\ \hline
$ m(\calg)$&2&5&12&29&70&169&408&985&2378&5741
\end{tabular}
\end{center}

\end{example}

\begin{remark} The relation between cluster algebras and continued fractions established in 
 Theorem \ref{thm1} has interesting implications on several levels. On the one hand, it allows us to use the machinery of continued fractions in the theory of cluster algebras to gain a better understanding of the elements of the cluster algebra and think of them in a new way. This aspect is explored in the  current paper.
 
 On the other hand, Theorem \ref{thm1} provides a new combinatorial model for continued fractions which allows us to interpret the numerators and denominators of positive continued fractions as cardinalities of  combinatorially defined sets. This interpretation is useful in the understanding of continued fractions. A simple example is that the continued fractions $\cfa$ and $[a_n,\ldots,a_2,a_1]$ have the same numerator; a fact which is well-known, but not obvious in the theory of continued fractions, but which becomes trivial using Theorem \ref{thm1}, because $\calg[a_n,\ldots,a_2,a_1]$ is obtained from $\calg\cfa$ by a rotation about 180$^\circ$.
Moreover, the snake graph calculus of \cite{CS,CS2,CS3} has produced a large number of identities of snake graphs which now can be translated to identities of (complex) continued fractions. In this way, we recover most of the classical identities involving convergents of the continued fractions, and also  obtain seemingly new identities. We give  three examples in section \ref{sect 4}, and we will  publish further results in this direction in a forthcoming paper \cite{CS5}.
\end{remark}
\medskip

\section{The set of all abstract snake graphs}\label{sect bijections}
Definition \ref{def} can be thought of as a map from the set of positive continued fractions to the set of snake graphs. On the other hand, the sign sequence (\ref{seq}) can be thought of as a map from snake graphs to positive continued fractions. Modulo a choice of an edge $e_d$ in the north-east of the snake graph, it turns out that these maps are bijections. We have the following theorem.

\begin{thm}
 \label{thm bijections} 
There is a commutative diagram

\[\xymatrix@C65pt{ 
\left\{\genfrac{}{}{0pt}{}{  \textup{\footnotesize pairs $(\calg,e_d)$ of a snake graph $\calg$ }}{\textup{\footnotesize with $d\ge1$ tiles, and $e_d\in\calgNE$ }}\right\}\ar@{>>}[dd]_{\textup{\footnotesize forget $e_d$}}  \ar@/^5pt/[r]^{\displaystyle\FF} _\cong&\left\{\genfrac{}{}{0pt}{}{\textup{\footnotesize positive continued fractions different}}{\textup{\footnotesize from  the continued fraction $[1] $ }}\right\}\ar@/^5pt/[l]^{\displaystyle\GG}
\ar@{>>}[dd]^{\displaystyle g} \\
 \\ 
\left\{{\genfrac{}{}{0pt}{}{\textup{\footnotesize snake graphs $\calg$ with}}{\textup{\footnotesize at least one tile}}}\right\}\ar[ddr]_{\displaystyle \chi}^
\cong  \ar[r]^{\displaystyle \FF'} &\left\{\genfrac{}{}{0pt}{}{\textup{\footnotesize positive continued fractions}}{\textup{\footnotesize with last coefficient $>1$}}\right\}
\ar[dd]^{\cong}_{\displaystyle Ev} \\
 \\
 & \mathbb{Q}_{> 1}
 }
 \]
 where  the maps are defined as follows.
 \begin{itemize}
\item [-]
$\FF$ maps the pair $(\calg,e_d)$ to the continued fraction   defined by the sign sequence 
 \[(f(e_0),f(e_1),\ldots,f(e_{d-1}),f(e_d)),\] 
  \item [-]$\FF'$ maps the snake graph  $\calg$ to the continued fraction defined by  the following sign sequence  taking the sign of the edge $e_{d-1}$ twice
  \[(f(e_0),f(e_1),\ldots,f(e_{d-1}),f(e_{d-1}));\]
  \item[-]$\GG$ sends the continued fraction $\cfa$ to the pair consisting of the snake graph $\calg\cfa$ and an edge $ e_d$ determined by the sign sequence (\ref{eqsign}),
\item[-] $g$ is  defined by 
 \[g ( [a_1,a_2,\ldots, a_{n-1},a_n]) =\left\{\begin{array}{ll} \,[a_1,a_2,\ldots, a_{n-1}+1] &\textup{if $a_n=1$;}\\ 
\,[a_1,a_2,\ldots, a_{n-1},a_n] &\textup{if $a_n>1$,}\end{array}\right. \]
\item[-] $\chi$ maps the snake graph $\calg$ to the quotient $\frac{m(\calg) } {m(\calg\setminus \calh_1)}$
and 
\item[-] $Ev$ is the bijection from Proposition \ref{thm162} sending the continued fraction to its value.

\end{itemize}
Moreover
the maps $\FF,\GG,\FF',\chi$ and $\textup{Ev}$ are  bijections.
\end{thm}

\begin{remark}
 We need to exclude the continued fraction $[1]$ in the statement of the theorem, because the map $g$ is not defined on it. Of course, it corresponds to the unique snake graph with zero tiles, the single edge snake graph.
\end{remark}
\begin{proof} Note first that all maps are well-defined, since the continued fraction $[1]$ is excluded.

 To show that the diagram commutes,
 suppose first that $f(e_d)=f(e_{d-1})$. Then $\FF(\calg,e_d)=\FF'(\calg)$. Moreover, if $\cfa=\FF(\calg,e_d)$ then $a_n>1$  and thus $g(\FF(\calg,e_d)) = \FF'(\calg)$.
Now suppose that $f(e_d)=-f(e_{d-1})$, and let $\cfa=\FF(\calg,e_d)$. Then $a_n=1$ and thus $g(\calf(\calg,e_d)) = [a_1,a_2,\ldots, a_{n-1}+1]$,
and this is also equal to $\FF'(\calg)$. 
The lower triangle commutes by Theorem \ref{thm1}.

Next, we show that $\FF$ and $\GG$ are bijections. Let $\calg$ be a snake graph with $d$ tiles, $e_d\in \calgNE$ and $f$ a sign function on $\calg$. Then $\FF(\calg,e_d)$ is the continued fraction $\cfa$ defined by the sign sequence 
\[(f(e_0),f(e_1),\ldots,f(e_{d-1}),f(e_d)),\] 
where $e_1,\ldots,e_{d-1}$ are the interior edges of $\calg$ and $e_0$ is the south edge of $\calg$. The map $\GG$ sends this continued fraction to the pair $(\calg\cfa,e'_d)$, where $\calg\cfa=\calg$ by construction, and $e'_d\in\calg\cfa^{N\!E}$ is the unique edge such that $f(e'_d)=f(e_d)$; thus $e'_d=e_d$. This shows that $\GG\circ\FF$ is the identity. Conversely, let $\cfa$ be a positive continued fraction. Then $\GG(\cfa)=(\calg\cfa,e_d)$, where $e_d$  is determined by the sign sequence (\ref{eqsign}). Applying $\FF$ to this pair yields the continued fraction determined by the same sign sequence. Thus $\FF\circ\GG$ is the identity.

To show that $\FF'$ is a bijection, let $\cfa$ be a positive continued fraction with $a_n>1$. Applying the map $\GG$ yields a pair $(\calg \cfa,e_d)$, where $e_d\in\calg\cfa^{N\!E}$ is the unique edge whose sign is equal to the sign of the last interior edge $e_{d-1}$, because $a_n>1$. Thus $f(e_d)=f(e_{d-1})=\sgn(a_n)$. Forgetting the edge $e_d$ and applying the map $\FF'$ will produce the continued fraction given by the sign sequence
\[(f(e_0),f(e_1),\ldots,f(e_{d-1}),f(e_{d-1})),\] 
and since $f(e_{d-1})=f(e_d)$, this continued fraction is $\cfa$. This shows that $\FF'$ is surjective. 
Now assume that there are two snake graphs $\calg,\calg'$, with respective sign functions $f,f'$ such that $\FF'(\calg)=\FF'(\calg')=\cfa$. Then the  sign sequences of $\calg$ and $\calg'$ have the same length $d+1=\sum a_j$, and either $f(e_i)=f'(e_i')$ for each $i=1,2,\ldots,d$, or $f(e_i)=-f'(e_i')$ for each $i=1,2,\ldots,d$. Since $e_0$, respectively $e_0'$, is the south edge of $\calg$, respectively $\calg'$, this implies that $\calg=\calg'$. Thus $\FF'$ is a bijection.

The map $Ev$ is a bijection by Proposition \ref{thm162}, and therefore the map $\chi$ is a bijection by commutativity of the diagram.
\end{proof}
\begin{figure}
\begin{center}
 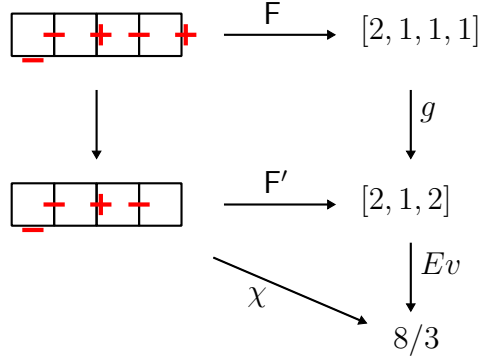
\caption{An example of the action of the maps in Theorem \ref{thm bijections}.} \label{fig bijections}
\end{center}
\end{figure}

\begin{cor} For every integer $N\ge 1$,
 the number of snake graphs $\calg$ with exactly $N$ perfect matchings is equal to   the Euler totient function  $\phi(N)$. 
\end{cor}
\begin{proof} The snake graph consisting of a single edge has $1=\phi(1)$ perfect matching. For $N\ge 2$, 
the map $\chi$ of the theorem induces  a bijection \[\{\calg \mid m(\calg)=N\}\to \{N/q \mid q\in\{1,2,\ldots, N-1 \}, \gcd(N,q)=1\},\] 
and the set on the right hand side has cardinality $\phi(N)$. 
\end{proof}

\begin{example}
 The 10 snake graphs with 11 perfect matchings are shown in Figure~\ref{fig 11} together with their rational numbers. 
 The corresponding continued fractions are the following. 
 
\[\begin{array}{llllllllllcccccc}
 \frac{11}{1}=[11] & \frac{11}{2}=[5,2] & \frac{11}{3}=[3,1,2]  
 & \frac{11}{4} =[2,1,3]
 & \frac{11}{5}=[2,5]\\ \\
  \frac{11}{6}=[1,1,5]
 & \frac{11}{7}=[1,1,1,3]
 & \frac{11}{8}=[1,2,1,2]
 & \frac{11}{9}=[1,4,2]
 & \frac{11}{10}=[1,10]
 \end{array}\]
\begin{figure}
\begin{center}
 \Large\scalebox{0.6}{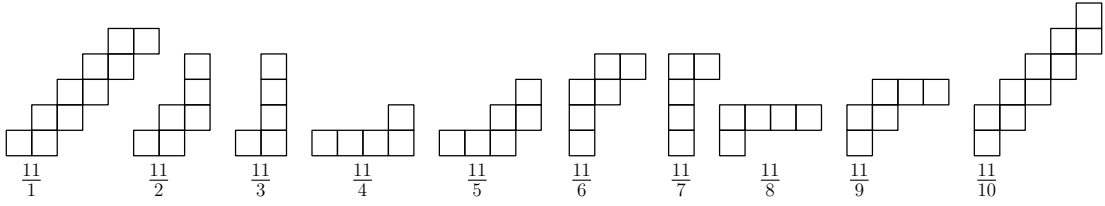}
\end{center}
\caption{The 10 snake graphs with 11 perfect matchings together with their rational numbers. }\label{fig 11}
\end{figure}
\end{example}

We also have the following analogue of Theorem \ref{thm bijections} for infinite snake graphs.
\begin{thm}\label{thm bijectionsinf}
 There is a commutative diagram
\[\xymatrix@C80pt{ 
\left\{\textup{infinite snake graphs}\right\}\ar[ddr]_{\displaystyle \chi}^
\cong  \ar[r]^{\displaystyle \FF} _\cong&\left\{\textup{infinite positive continued fractions}\right\}
\ar[dd]^{\cong}_{\displaystyle Ev} \\
 \\
 & \mathbb{R}_{> 1}\setminus\mathbb{Q}_{> 1}
 }
 \]
 where  the maps are defined as follows.
 \begin{itemize}
\item [-]
$\FF$ maps the  snake graph  $\calg$ to the continued fraction defined by  the infinite sign sequence  
  \[(f(e_0),f(e_1),f(e_2)\ldots);\]
\item[-] $\chi$ maps the snake graph $\calg$ to the limit \[\lim_{n\to \infty}\frac{m(\calg\cfa) } {m(\calg\cfa\setminus \calh_1)};\]
\item[-] $Ev$ is the bijection from Proposition \ref{thm162R} sending the continued fraction to its value.

\end{itemize}
Moreover
the maps $\FF,\chi$ and $\textup{Ev}$ are  bijections.

\end{thm}

%
%
\section{Three formulas for continued fractions}\label{sect 4}
In this section, we give three examples of formulas that follow directly from our  earlier work on snake graph calculus. We restrict ourselves here to the  three cases that are easiest to state in terms of continued fractions. However, the results in \cite{CS,CS2} yield several other formulas, which  
 will be given in a forthcoming paper.
  
In what follows, the identities of snake graphs are identities in the snake ring introduced in \cite{CS3}. Equivalently, one can think of these identities as canonical bijections between the sets of perfect matchings of the snake graphs involved, where multiplication is to be replaced by disjoint union of graphs and addition by union of sets of perfect matchings. For example an identity of the form $\calg_1\calg_2=\calg_3\calg_4+\calg_5\calg_6$ means that there is a bijection between $\match (\calg_1\sqcup\calg_2) $ and $\match(\calg_3\sqcup\calg_4)\cup\match(\calg_5\sqcup\calg_6)$, where $\match(\calg)$ denotes the set of perfect matchings of $\calg$.
\smallskip

We keep the notation of section \ref{sect 1}.
Let $b_{i}$  be the single edge corresponding to the tile $G_{\ell_i}$ in $\calg=\calg\cfa$ and let $b_0$ be the unique edge in $\calgSW\setminus\{e_0\}$ and $ b_n$ the unique edge in $\calgNE\setminus\{e_d\}$.
\begin{thm}
 \label{grafting}
 We have the following identities of snake graphs, where we set $\calg[a_1,\ldots,a_0]= b_0$ and $\calg[a_{n+1},\ldots,a_n]=b_n$.
 \begin{itemize}
\item[\textup{(a)}] For every $i=1,2,\ldots, n$,
\[b_{i} \,\calg[a_1,\ldots,a_n]\  =\  \calg[a_1,\ldots,a_i] \,\calg[a_{i+1},\ldots,a_n] + \calg[a_1,\ldots,a_{i-1}]\, \calg[a_{i+2},\ldots,a_n].\]
\\
 \item[\textup{(b)}] 
 For every $j\ge 0$ and $i$ such that $1\le i+j\le n-1$,
 \[\begin{array}{rcl} \calg[a_1,\ldots,a_{i+j}]  \calg[a_i,\ldots,a_n] &=& \calg[a_1,\ldots,a_n]\, \calg[a_{i},\ldots,a_{i+j}] \\  
 \\
 &+& (-1)^j \,\calg[a_1,\ldots,a_{i-2}] \,\calg[a_{i+j+2},\ldots,a_n].\end{array} \]
\\
 \item[\textup{(c)}] For snake graphs $\calg[a_1,\ldots,a_n]$ and $\calg[b_1,\ldots,b_m]$ such that $[a_i,\ldots,a_{i+k}]=[b_j,\ldots,b_{j+k}]$ for $k\ge0,1<i<n, 1<j<m$ and $a_{i\!-\!1} < b_{j\!-\!1}, a_{i\!+\!k\!+\!1}\neq b_{j\!+\!k\!+\!1},$ we have 
\[
\begin{tikzpicture}
\node at (0,1){$\calg[a_1,\ldots,a_{n}] \, \calg[b_1,\ldots,b_m] = \calg[a_1,\ldots,a_{i\!-\!1},b_{j}, \ldots,b_m] \, \calg[b_1,\ldots,b_{j\!-\!1},a_i,\ldots,a_{n}] $};
\node at (2.5,0){$+(-1)^k \,\calg[a_1,\ldots,a_{i-2}\!-\!1,1,b_{j\!-\!1}-a_{i\!-\!1}\!-\!1,b_{j\!-\!2},\ldots,b_1] \,\calg'$};
\end{tikzpicture} 
\]
where 
\[\calg'=\left\{
\begin{array}{ll}
\calg[b_m,\ldots,b_{j\!+\!k\!+\!2}\!-\!1,1,a_{i\!+\!k\!+\!1}-b_{j\!+\!k\!+\!1}\!-\!1,a_{i\!+\!k\!+\!2},\ldots,a_n] & \textup{if } a_{i\!+\!k\!+1}>b_{j\!+\!k\!+1} \\
\\
\calg[b_m,\ldots,b_{j\!+\!k\!+\!2},b_{j\!+\!k\!+\!1}\!-a_{i\!+\!k\!+\!1}\!-\!1,1,a_{i\!+\!k\!+\!2}\!-\!1,,\ldots,a_n] & \textup{if } a_{i\!+\!k\!+1}<b_{j\!+\!k\!+1}
\end{array} 
\right.
\]
\end{itemize}
\end{thm}

\begin{proof}
 Part (a) follows from a construction called grafting with the single edge $b_i$, see \cite[section 3.3, case 3]{CS2}. Part (b) is the resolution of two snake graphs with a crossing overlap $\calg[a_i,\ldots,a_{i+j}]$ \cite[section 2.4]{CS}. If $j $ is even, then the overlap is crossing in the pair on the left $\calg[a_1,\ldots,a_{i+j}]  \calg[a_i,\ldots,a_n]$, and if $j$ is odd, then the overlap is crossing in the pair on the right $ \calg[a_1,\ldots,a_n] \calg[a_{i},\ldots,a_{i+j}]$.  Part (c) is the resolution of two snake graphs with a crossing overlap $\calg[a_{i\!-\!1},a_i,\ldots,a_{i+k},d]=\calg[a_{i\!-\!1},b_j,\ldots,b_{j+k},d]$ where $d= \textup{min } \{ a_{i\!+\!k\!+\!1}, b_{j\!+\!k\!+\!1} \}$, see \cite[ section 3.1]{CS2}.
\end{proof}

Translating  into continued fractions, we get the following. 
\begin{thm}
 \label{graftingcf}
 We have the following identities of numerators of continued fractions, where we set $\caln[a_1,\ldots,a_0]=1,$ and $\caln[a_{n+1},\ldots,a_n]=1$.

 \begin{itemize}
\item[\textup{(a)}] For every $i=1,2,\ldots, n$, \[\caln[a_1,\ldots,a_n] = \caln[a_1,\ldots,a_i] \, \caln[a_{i+1},\ldots,a_n] + \caln[a_1,\ldots,a_{i-1}]\, \caln[a_{i+2},\ldots,a_n]. \]
 \item[\textup{(b)}] 
 For every $j\ge 0$ and $i$ such that $1\le i+j\le n-1$,
 \[\begin{array}{rcl}  \caln[a_1,\ldots,a_{i+j}]\,  \caln[a_i,\ldots,a_n]&=& \caln[a_1,\ldots,a_n]\, \caln[a_{i},\ldots,a_{i+j}] \\
 \\
  &+&(-1)^j  \,\caln[a_1,\ldots,a_{i-2}] \,\caln[a_{i+j+2},\ldots,a_n]. \end{array}\]
  \\
 \item[\textup{(c)}] For continued fractions $[a_1,\ldots,a_n]$ and $[b_1,\ldots,b_m]$ such that $[a_i,\ldots,a_{i+k}]=[b_j,\ldots,b_{j+k}]$ for ${0\le k \le \textup{min}(n-i-2, m-j-2),  \ 2<i<n, \ 2<j<m}$ and $a_{i\!-\!1} < b_{j\!-\!1}, a_{i\!+\!k\!+\!1}\neq b_{j\!+\!k\!+\!1},$ we have 
\[
\begin{tikzpicture}
\node at (0,1){$\caln[a_1,\ldots,a_{n}] \, \caln[b_1,\ldots,b_m] = \caln[a_1,\ldots,a_{i\!-\!1},b_{j}, \ldots,b_m] \, \caln[b_1,\ldots,b_{j\!-\!1},a_i,\ldots,a_{n}] $};
\node at (2.5,0){$+(-1)^k \,\caln[a_1,\ldots,{a_{i-2}} \!-\!1,1,b_{j\!-\!1}-a_{i\!-\!1}\!-\!1,{b_{j\!-\!2}},\ldots,b_1] \,\caln'$};
\end{tikzpicture} \]
where 
\[\caln'=\left\{
\begin{array}{ll}
\caln[b_m,\ldots,b_{j\!+\!k\!+\!2}\!-\!1,1,a_{i\!+\!k\!+\!1}-b_{j\!+\!k\!+\!1}\!-\!1,a_{i\!+\!k\!+\!2},\ldots,a_n] & \textup{if } a_{i\!+\!k\!+1}>b_{j\!+\!k\!+1}; \\
\\
\caln[b_m,\ldots,b_{j\!+\!k\!+\!2},b_{j\!+\!k\!+\!1}\!-a_{i\!+\!k\!+\!1}\!-\!1,1,a_{i\!+\!k\!+\!2}\!-\!1,,\ldots,a_n] & \textup{if } a_{i\!+\!k\!+1}<b_{j\!+\!k\!+1}.
\end{array} 
\right.
\]
\end{itemize}\end{thm}
\begin{proof}
 (a) The equation in part (a) of Theorem \ref{grafting} is an equation of Laurent polynomials in several variables $x_i$. Setting each of these variables equal to 1, we get the identity
\[\begin{array}{cc}m(\calg[a_1,\ldots,a_n])=  m(\calg[a_1,\ldots,a_i]) \,m(\calg[a_{i+1},\ldots,a_n])+ m(\calg[a_1,\ldots,a_{i-1}])\, m(\calg[a_{i+2},\ldots,a_n]).
 \end{array}\]
Now the result follows from Theorem \ref{thm1}. Part (b)  and (c) are proved in the same way.
\end{proof}

\begin{remark} Parts (a) and (b) are known formulas for continuants, see \cite[Chapter XIII]{Muir} or \cite{ustinov} for (b).

 We did not find the result (c) in the literature. An example of this identity is the following, where $i=j=3, k=1$.
 \[ \caln[1,2,3,4,3]\,\caln[1,4,3,4,1,2] = \caln[1,2,3,4,1,2]\,\caln[1,4,3,4,3]-\caln[{0,1,1,1}]\,\caln[1,1,1]\]
 The left hand side is equal to $139\cdot 239$, the right hand side is equal to $149\cdot 223 - 2\cdot 3$,  and both are equal to 33221.
\end{remark}

\begin{remark}
 The formulas in Theorem \ref{grafting} are identities for elements of the cluster algebra. This implies that in the formulas in Theorem \ref{graftingcf}, we can replace the integers $a_i$ by certain Laurent polynomials as shown in section \ref{sect cluster} (using Theorem~\ref{grafting}).  
We shall then show in section \ref{sect 62} that the formulas in Theorem \ref{graftingcf} also hold for complex numbers $a_i$.
\end{remark}

 \section{The relation to cluster  algebras} \label{sect cluster}
We have seen that the numerator of a continued fraction is equal to the number of perfect matchings of the corresponding abstract snake graph, and that it can therefore be interpreted as the number of terms in the numerator of the Laurent expansion of an associated cluster variable. In other words, if one knows the Laurent polynomial of the cluster variable, one obtains the numerator of the continued fraction by setting all variables $x_i$ (and coefficients $ y_i$) equal to 1. 

In this section, we show how to go the other way and recover the Laurent polynomials of the cluster variables from the continued fractions. More precisely, we will associate a Laurent polynomial $L_i$ to each $a_i$ in such a way that  the continued fraction $[L_1,\ldots, L_n]$ of Laurent polynomials is equal to  a quotient of cluster variables.

Let $(S,M)$ be a marked surface with a triangulation $T$ that has no self-folded triangles. Denote by $\cala=\cala(S,M)$ the associated cluster algebra with initial cluster $\mathbf{x}_T=(x_1,x_2,\ldots,x_N)$ corresponding to $T$. 
Let $\zg$ be a generalized arc in $(S,M)$, and let $\calg_\zg$ be its (labeled) snake graph as defined in section \ref{sect label}. Using Theorem \ref{thm bijections}, there is a positive continued fraction $\cfa$ such that $\calg_\zg=\calg\cfa$ as graphs. Let $\ell_i=\sum_{j=1}^i a_j$ and $\calh_1,\calh_2,\ldots,\calh_n$ be the subgraphs as defined in section \ref{sect 1}. For $i=1,2,\ldots,n-1$, denote by $b_i$ the label of the tile $G_{\ell_i}$ and let $b_0$ be the unique
 edge in $\calgSW\setminus\{e_0\}$ and $b_n$ be the unique edge in $\calgNE\setminus\{e_d\}$. We define the following Laurent polynomials, where  we simply write $b_i$ for the cluster variable $x_{b_i}$.

\begin{definition}\label{def L}
 
  \[\begin{array}{rcl}
L_1 &=&\displaystyle  x(\calh_1)\,\frac{1}{b_1}\\ \\ 
L_2 &=&\displaystyle  x(\calh_2) \,\frac{b_1}{b_0b_2}\\ \\
L_3 &=& \displaystyle x(\calh_3) \,\frac{b_0b_2}{b_1^2b_3}\\ \\
 L_i &=&  \left\{\begin{array}{ll} \displaystyle x(\calh_i)\,\frac{b_0b_2^2b_4^2\cdots b_{i-3}^2 b_{i-1}} {b_1^2b_3^2\cdots b_{i-2}^2b_{i}} &\textup{if $i$ is odd;}\\
 \\
              \displaystyle x(\calh_i) \,\frac {b_1^2b_3^2\cdots b_{i-3}^2b_{i-1}} {b_0b_2^2b_4^2\cdots b_{i-2}^2 b_{i}}&\textup{if $i$ is even.}
              \end{array}\right.
              \end{array}
\]
\end{definition}

The following explicit formula for the $x(\calh_i)$ is taken from  \cite{CF}. We include a proof here for convenience of the reader.
\begin{prop} With the above notation
 \label{prop hi}
 \[x(\calh_i)=x_{\ell_{i-1}} \left( \sum_{j=\ell_{i-1}}^{\ell_{i}-1} \frac{x_{e_j}}{x_jx_{j+1}}\right)x_{\ell_{i}},\]
where $\ell_0=0,x_{0}=b_0$ and $x_{\ell_d}=b_n$.
\end{prop}
\begin{proof}
 We prove the formula for $i=1$. For arbitrary $i$, the proof is the same up to a shift in the subscripts. 
 If $a_1=1$ then $\calh_1$ is a single edge $e_0$ and thus $x(\calh_1)=x_{e_0}$. On the other hand, $\ell_1=\ell_0+1$ and thus the formula in the proposition becomes 
 $x_{0}\frac{x_{e_0}}{x_0 x_1} x_1=x_{e_0}$, proving the result for $a_1=1$.
 
 Now suppose $a_1>1$.
 According to Lemma \ref{lem2},
 $\calh_1=\calg[a_1]$ is a zigzag snake graph with $a_1-1$ tiles $G_1,G_2,\ldots,G_{a_1-1}$ and $a_1$ perfect matchings. An example with $a_1=7$ is given in Figure \ref{fig hi}. 
\begin{figure}
\begin{center}
 \scriptsize 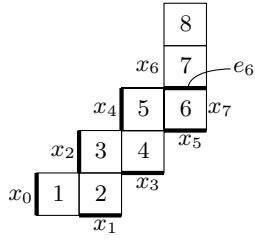
\end{center} 
\caption{An example of the graph $\calh_1$ for $a_1=\ell_1=7$. $\calh_1$ consists of the first 6 tiles. The perfect matching indicated on $\calh_1$ is the unique matching containing the edge $e_6$.}\label{fig hi}
\end{figure}
 Note that $\ell_1=a_1$. Let $\calh_1'$ be the subgraph consisting of the tiles  $G_1,G_2,\ldots,G_{a_1-2}$. Then $\calh_1'$ has $a_1-1$ perfect matchings each of which can be completed to a perfect matching of $\calh_1$ by adding the unique edge $e\in \calh_1^{N\!E}$ that is not equal to $e_{\ell_1-1}$. Note that the edge $e$ has label $\ell_1$. The only remaining perfect matching $P_0$ of $\calh_1$ does not contain the edge $e$ and it does contain the edge $e_{\ell_1-1}$ and every other boundary edge of $\calh_1$. In the example in Figure \ref{fig hi}, the matching $P_0$ is marked in black. Since $\calh_1$ is a complete zigzag snake graph, the weight of $P_0$ is $x(P_0)=x_0x_1x_2\cdots x_{\ell_1-2}x_{e_{\ell_1-1}}$, and thus
 \[\frac{x(P_0)}{x_1x_2\cdots x_{\ell_1-1}} =\frac{x_0 x_{e_{\ell_1-1}}}{x_{\ell_1-1}}.\]
 Therefore, 
 \[x(\calh_1)=x(\calh_1')\frac{x_{\ell_1}}{x_{\ell_1-1}} + \frac{x_0 x_{e_{\ell_1-1}}}{x_{\ell_1-1}},\]
 and by induction, this is equal to
 \[  x_0    \left( \sum_{j=0}^{\ell_{1}-2} \frac{x_{e_j}}{x_jx_{j+1}}\right)x_{\ell_{1}-1}    \frac{x_{\ell_1}}{x_{\ell_1-1}} + \frac{x_0 x_{e_{\ell_1-1}}}{x_{\ell_1-1}}
 =
 x_0    \left( \sum_{j=0}^{\ell_{1}-1} \frac{x_{e_j}}{x_jx_{j+1}}\right)x_{\ell_{1}}  \qedhere
 \]
\end{proof}
\medskip

To state our main result, we need to construct a second generalized arc $\zg'$ in such a way that its snake graph is obtained from the snake graph $\calg_\zg$ of $\zg$ by removing the first zigzag subgraph $\calh_1$. For an example see Figure \ref{fig exL}.
%

  Let $s$ be the starting point of the generalized arc $\zg$, and let $\zD$ be the first triangle in $T$ that $\zg$ meets. 
Then the two sides of $\zD$ that are incident to $s$ correspond to the two edges in $\calgSW_\zg$. Let $\za$ be the side that corresponds to the edge in $\calgSW_\zg$ whose sign is opposite to the sign of the first interior edge of $\calg_\zg$.  In other words, $\za = e_0$ if $a_1=1$ and $\za= b_0$ if $a_1>1$. We orient $\za$ such that its terminal point is the point $s$, and we consider the generalized arc $\zg'$ corresponding to the concatenation $\zg\circ\za$. Since arcs are considered up to isotopy, it follows from the construction that $\zg'$ avoids the first $a_i$ crossings of $\zg$ with $T$ and then runs parallel to $\zg$, see Figure~\ref{fig exL}. On the level of snake graphs we have $\calg_{\zg'}=\calg_\zg\setminus\calh_1=\calg[a_2,\ldots,a_n]$.

\medskip

We are now ready for the main theorem of this section. 
\begin{thm}
 \label{thmx}
Let $\zg$ be a generalized arc in $(S,M,T)$ and let $\calg_\zg=\calg[a_1,a_2,\ldots,a_n]$ be its snake graph. Then we have the following identity in the field of fractions of the cluster algebra $\cala(S,M)$
\[\frac{x_\zg}{x_{\zg'}} =[L_1,L_2,\ldots,L_n],\]
where $\zg'$ is the generalized arc constructed above and $L_i$ the Laurent polynomials in Definition~\ref{def L}. Moreover, $x_\zg$ and $x_{\zg'}$ are relatively prime in the ring of Laurent polynomials and thus the left hand side of the equation is reduced.

 In particular, if $\zg$ and $\zg'$ have no self-crossing, then left hand side of the equation is the quotient of two cluster variables.
 
\end{thm}
\begin{remark}
 It is possible that $\zg$ has no self-crossings but $\zg'$ does. For an example see Theorem \ref{thm inf}.
\end{remark}
\begin{remark}
 Theorem \ref{thmx} suggests that one should study quotients of elements of the cluster algebra and their asymptotic behavior. We present a first result in this direction in section~\ref{sect 5}.
\end{remark}

\begin{proof} 
If $n=1$ then $\zg'=b_1$ by construction. Therefore $x_\zg/x_{\zg'}=x(\calg_\zg)/b_1 =x(\calh_1)/b_1=L_1$. 
If $ n=2$, then 
\[ [L_1,L_2]=L_1+\frac{1}{L_2}= \frac{L_1L_2+1}{L_2}= \frac{x(\calh_1) x(\calh_2)+b_0b_2}{b_1 \,x(\calh_2)}= \frac{x(\calg[a_1]) x(\calg[a_2])+b_0b_2}{b_1 \,x(\calg[a_2])}, \]
and using Theorem \ref{grafting}(a) with $i=1$ and $n=2$ this is equal to
$ x(\calg[a_1,a_2])/x(\calg[a_2]) = x_{\zg}/x_{\zg'}$.

Now suppose that $ n>2$. Again  using Theorem \ref{grafting}(a), we have
\[ b_1\, \calg\cfq = \calh_1 \,\calg[a_2,a_3, \ldots,a_n] +b_0\, \calg[a_3,a_4,\ldots,a_n].\]
Passing to the Laurent polynomials and 
dividing by $x(\calg[a_2,a_3, \ldots,  a_n])$ and $b_1$ we get
\[\frac{x_\zg}{x_{\zg'}}=
\frac{x(\calg\cfq)}{x(\calg[a_2,a_3,\ldots,a_n])}  = \frac{x(\calh_1 )}{b_1 }+\frac{1}{ \frac{b_1}{b_0} \frac{x(\calg[a_2,a_3, \ldots,a_n])}{ x(\calg[a_3,a_4,\ldots,a_n])}}.\]
Using the definition of $L_1$ and induction, this is equal to 
\begin{equation}\label{eq lem cf} L_1+\frac{1}{ \frac{b_1}{b_0} [L_2',L_3',\ldots,L_n'] }\end{equation}
where $L_i'= \frac{b_0}{b_1} L_i$ if $i$ is even, and  $L_i'= \frac{b_1}{b_0} L_i$ if $i$ is odd.
Lemma \ref{lem cf} implies that (\ref{eq lem cf}) is equal to 
\[ L_1+\frac{1}{ [L_2,L_3,\ldots,  L_n]}\]
which in turn equals $\cfl$.
Finally, if a Laurent polynomial $f$ is a common divisor of $x_\zg$ and $ x_{\zg'}$ then the number of terms in $f$ is a common divisor of the number of terms in $x_\zg$ and $x_{\zg'}$, and Theorem \ref{thm1} implies that $f$ is a Laurent monomial. Thus $f$ is a unit in the ring of Laurent polynomials.
%
%
%
This completes the proof.
 \end{proof}
 
Since $x_\zg$ and $x_{\zg'}$ are relatively prime, we obtain the following formula for $x_\zg$.
\begin{cor}
 With the notation of Theorem \ref{thmx}, we have 
 \[x_\zg=\caln\,\cfl \cdot M,\]
 for some Laurent monomial $M$.
\end{cor}

\begin{remark} Theorem \ref{thmx} provides a very efficient way to compute the Laurent expansion of a  cluster variables $x_\zg$, without determining the perfect matchings of the snake graph $\calg_\zg$.
\end{remark}

\begin{example}\label{ex 6.6} To illustrate the theorem, we compute the quotient of the cluster variables associated to the arcs $\zg$ and $\zg'$ in the triangulation shown on the left of Figure \ref{fig exL}. 
\begin{center}
\begin{figure}
  \tiny\scalebox{1.15}{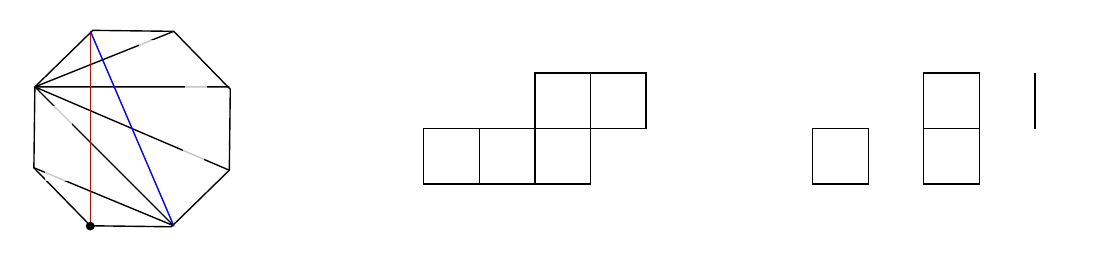} 
\caption{Example \ref{ex 6.6}}\label{ExCondFracSG}
\label{fig exL}\end{figure}
\end{center}The snake graph $\calg_\zg$ of $\zg$ is shown in the center of Figure \ref{ExCondFracSG}. If we choose $e_5\in\calgNE$ to be the east edge then $\calg$ corresponds to the continued fraction $[2,3,1]$ and the corresponding subgraphs $\calh_1,\calh_2,\calh_3$ are shown on the right of the figure. By definition, we have $b_1=x_2, b_2=x_5$ and 
\[L_1=\frac{x(\calh_1)}{x_2}, L_2=\frac{x(\calh_2) x_2}{ b_0x_5}, L_3=\frac{x(\calh_3) b_0x_5}{x_2^2 b_3}.\]
Thus 
\[\begin{array}{rcl}[L_1,L_2,L_3]&=& (L_1L_2L_3+L_1+L_3)/(L_2L_3+1)\\
\\
&=& \displaystyle\frac{\frac{ \displaystyle x(\calh_1)x(\calh_2)x(\calh_3)}{  \displaystyle x_2^2b_3}+\frac{ \displaystyle x(\calh_1)}{ \displaystyle x_2}+\frac{ \displaystyle x(\calh_3)b_0x_5}{ \displaystyle x_2^2  b_3}} {\frac{ \displaystyle x(\calh_2)x(\calh_3) +x_2b_3}{ \displaystyle x_2b_3}}\\
\end{array}\]
and expanding by $\frac{x_2b_3}{x_5}$ shows that this quotient is equal to 

\[  \frac{ \frac{ \displaystyle x(\calh_1)x(\calh_2)x(\calh_3)}{ \displaystyle x_2x_5}+\frac{ \displaystyle x(\calh_1)x_2x_3x_4b_3}{ \displaystyle x_2x_3x_4x_5}+\frac{ \displaystyle x(\calh_3)b_0x_1x_3x_4x_5}{ \displaystyle x_1x_2x_3x_4x_5}}{\frac{ \displaystyle x(\calh_2)x(\calh_3)}{ \displaystyle x_5} +\frac{ \displaystyle x_2x_3x_4b_3}{ \displaystyle x_3x_4x_5}}\] 

On the other hand, \[x_\zg=x(\calg_\zg)=\frac{1}{x_1x_2x_3x_4x_5}\ { \displaystyle \sum_{P\in\Match\calg_\zg}} x(P),\] where the sum is over all perfect matchings of $\calg_\zg$.
The snake graph $\calg_\zg$ has precisely 9 perfect matchings which are shown in Figure \ref{fig exL2}.
The two perfect matchings on the left correspond to the term $x(\calh_1)x_2x_3x_4b_3/x_2x_3x_4x_5$, the six perfect matchings in the center correspond to the term $ x(\calh_1)x(\calh_2)x(\calh_3)/x_2x_5$ and the perfect matching on the right corresponds to the term $x(\calh_3)b_0x_1x_3x_4x_5/x_1x_2x_3x_4x_5$.
\begin{center}
\begin{figure}
  \scalebox{0.8}{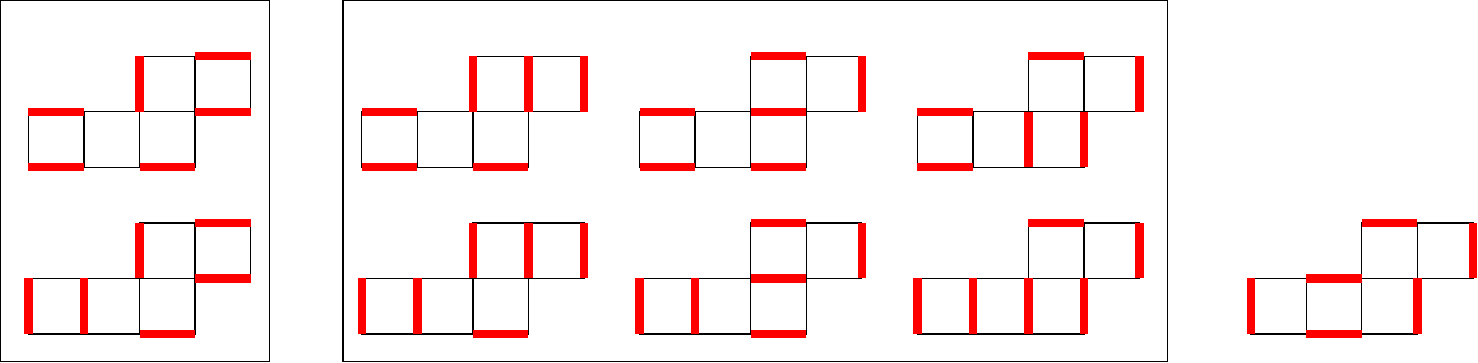} 
\caption{The 9 perfect matchings of $\calg_\zg$ in Example \ref{ex 6.6}}
\label{fig exL2}\end{figure}
\end{center}
Similarly, the snake graph $\calg_{\zg'}$ of $\zg'$ consists of the 3 tiles labeled 3, 4 and 5 and has four perfect matchings.  Three of these perfect matchings correspond to the term $x(\calh_2)x(\calh_3)x_5$ and the other to the term $x_2x_3x_4b_3/x_3x_4x_5.
$ 

Thus we see that \[\frac{x_\zg}{x_{\zg'}} = [L_1,L_2,L_3].\]

\smallskip
Alternatively, according to Theorem \ref{thm bijections}, we can also choose $e'_5\in\calgNE$ to be the north edge, and then $\calg$ corresponds to the continued fraction $[2,4]$. If we denote by $\calh_i',b_i'$ the corresponding subgraphs and by $L_i'$ the Laurent monomials, then  we have 
\begin{itemize}
\item[-] $\calh_1'=\calh_1$ and $\calh_2'$ consists of the 3 tiles labeled 3,4,5;
\item[-] $b_0'=b_0, b_1'=x_2$, and $b_2'=e_5$ is the east edge of $\calgNE$;
\item[-] $L_1'=L_1$ and \[L_2'=x(\calh'_2)\frac{b_1'}{b_0'b_2'}=x(\calh'_2)\frac{x_2}{b_0e_5}.\]
\end{itemize}

Therefore 
\[ [L_1',L_2']= \frac{L_1'L_2'+1}{L_2'} = \left.\left(\frac{x(\calh_1)x(\calh_2')}{x_2} +\frac{b_0 e_5x_1x_3x_4x_5}{x_2x_1x_3x_4x_5}\right)\right/ x(\calh_2'),\]
which again is equal to $x_\zg/x_\za$.
In Figure \ref{fig exL2}, the 2 perfect matchings in the first frame together with the 6 perfect matchings in the second frame correspond to the term $\frac{x(\calh_1)x(\calh_2')}{x_2}$ in the expression above, whereas the single perfect matching on the right of Figure~\ref{fig exL2} corresponds to the term $\frac{b_0 e_5x_1x_3x_4x_5}{x_2x_1x_3x_4x_5}.$
\end{example}

\begin{remark}
 The proof of Theorem~\ref{thmx} actually shows a stronger result. Namely that 
 \[\frac{x(\calg\cfa)}{x(\calg[a_2,\ldots,a_n])} =[L_1,L_2,\ldots,L_n]\] 
 for every labeled snake graph $\calg\cfa $ whose labels satisfy the condition that for every straight subsnake graph of three consecutive tiles $G_{b_i-1},G_{b_i},G_{b_i+1}$ 
\begin{enumerate}
\item   the north edge of $ G_{b_i-1}$ and the south edge of $G_{b_i+1}$ both have the same weight $b_i$, if the subsnake graph is horizontal;
\item   the east edge of $ G_{b_i-1}$ and the west edge of $G_{b_i+1}$ both have the same weight $b_i$, if the subsnake graph is vertical.
\end{enumerate}
\end{remark}

\subsection{Complex continued fractions}\label{sect 62}
Theorem \ref{thmx} allows us to consider continued fractions of complex numbers as follows. Let $[z_1,z_2,\ldots,z_n]$ be a continued fraction of complex numbers $z_j\in\mathbb{C}$.
Define $a_j=\lceil \,|z_j |\, \rceil $ to be the least integer larger or equal to the absolute value $|z_j|$ of $z_j$. Thus 
$\cfa$ is a positive continued fraction and we can associate to it the snake graph $\calg\cfa$. Let $d$ be the number of tiles as usual. We can realize this snake graph as the snake graph of an arc $\zg$ in a triangulated polygon $(S,M,T)$ such that each edge of the snake graph corresponds to an arc in $T$ and hence to a cluster variable. It is important to realize that, since we are working in a polygon, the label of the edge $e_j$, for $j=0,1,\ldots, d$, is different from the label of any other edge of $\calg\cfa$. 

We 
 could now compute the quotient of cluster  variables $x_\zg/x_{\zg'}$ as in Theorem~\ref{thmx}. 
However, we will  use the following specialization of the cluster variables. Recall that the integers $\ell_i=\sum_{j=1}^i a_j$ and the subgraphs $\calh_i$ are defined in such a way that the interior edge $e_{\ell_i-1}$ lies in $\calh_i^{N\!E}$. We set $x_{e_{\ell_j-1} }= z_i-a_i+1$, for $j=1,2,\ldots,n$ and $x_e=1$, for all other edges $e$ in $\calg\cfa$. Then Proposition~\ref{prop hi} implies that under this specialization
\[x(\calh_i)=\sum_{j=\ell_{i-1}}^{\ell_i-1} x_{e_j} =x_{e_{\ell_i-1} }+\ell_i-\ell_{i-1}-1 
= z_i-a_i+1+a_i-1=z_i.\]
Moreover, under this specialization all $b_i$ have label 1, and thus $L_i=x(\calh_i)=z_i$. Therefore
\[\caln[z_1,z_2,\ldots,z_n]=\caln[L_1,L_2,\ldots,L_n],\]
where the left hand side is the continuant of $[z_1,z_2,\ldots,z_n]$.
This shows that the formulas of Theorem \ref{graftingcf} also hold for continued fractions of complex numbers.
\begin{example} 
 If $[z_1,z_2]=[2i,-3+i]$, we have $a_1=2$, $a_2= \lceil \sqrt{10}\,\rceil =4$, and the special weights are $x_{e_1}=2i-1$ and $ x_{e_6}= -6+i$. Then $L_1=2i$, $L_2= -3+i$, so $\caln[L_1,L_2]=2i(-3+i)+1= z_1z_2+1=\caln[z_1,z_2]$.
\end{example}

\subsection{An application to quiver Grassmannians}
We keep the notation of Theorem \ref{thmx}. Let $\zL_T$ be the Jacobian algebra (over $\mathbb{C}$) of the quiver with potential determined by the triangulation $T$, see \cite{Labardini}, and let $M_\zg$ be the indecomposable $\zL$-module corresponding to $\zg$. From now on assume that $\zg$ has no self-crossing. Then $M_\zg$ is a rigid module and $x_\zg$ is a cluster variable. Then the cluster character formula  \cite{Pa} implies that 
\begin{equation}
 \label{eqgr} 
 x_\zg = \sum_{\underline e} \chi(\mathrm{Gr}_{\underline e} (M_\zg)) \prod_{i=1}^N x_i^{a_i(\underline{e})},
\end{equation}
for some ${a_i(\underline{e})}\in \mathbb{Z}$, where the sum runs over all dimension vectors $\underline{e}\in \mathbb{Z}_{\ge 0}^N$, $\mathrm{Gr}_{\underline e} (M_\zg) $ denotes the quiver Grassmannian of $M_\zg$ of submodules of dimension vector $\underline{e}$, and $\chi$ is the Euler characteristic. The positivity theorem for cluster variables implies that $ \chi(\mathrm{Gr}_{\underline e} (M_\zg))\ge 0$. 

Let $\calg_\zg$ be the snake graph of $\zg$ and $\cfa$ be its continued fraction.
Comparing the formula (\ref{eqgr}) to the perfect matching formula of \cite{MSW},  and specializing all $x_i=1$, we see that the number of perfect matchings of $\calg_\zg$ is equal to the sum  $\sum_{\underline e} \chi(\mathrm{Gr}_{\underline e} (M_\zg)$. Now Theorem~\ref{thm1} implies the following.
\begin{cor}
 \label{corgr} 
 With the notation above
\[\sum_{\underline e} \chi(\mathrm{Gr}_{\underline e} (M_\zg))= \caln\cfa. \]
\end{cor}
 
 \section{Infinite continued fractions and limits of rational functions from cluster algebras} \label{sect 5}
 If $[a_1,a_2,\ldots ]$ is an infinite positive continued fraction, we  associate to it an infinite snake graph $\calg[a_1,a_2,\ldots ]$ which is defined in the same way as in section~\ref{sect 1}. Of course this infinite snake graph has infinitely many perfect matchings, and  the numerator of the continued fraction is also infinite. However, according to Proposition \ref{thm162R}, the continued fraction converges  and its limit is a real irrational number. Therefore Theorem \ref{thm1} implies that
 
\[ [a_1,a_2,\ldots ] =\lim_{n \to \infty} \frac{m(\calg[a_1,a_2,\ldots,  a_n])}{m(\calg[a_2,a_3,\ldots,  a_n])}.\]

\begin{example}
 The limit of the ratios of  consecutive Fibonacci numbers is the golden ratio. Thus
 \[[\,\overline{1}\,]=[1,1,1,\ldots ] =\frac{1+\sqrt{5}}{2}.\]
 The limit of the ratios of  consecutive Pell numbers is the silver ratio. Thus
  \[[\,\overline{2}\,]= [2,2,2,\ldots ] =1+\sqrt{2}.\]
 
\end{example}
 
 Using our construction of section \ref{sect cluster}, we can associate to  a  continued fraction $[a_1,a_2, \ldots] $ the rational function $[L_1,L_2,\ldots ]$. If each of the finite snake graphs $\calg\cfa$ has a geometric realization in the same cluster algebra, this rational function represents a limit of quotients of elements of the cluster algebra.
 
 For example, consider the torus with one puncture and denote by $\x=(x_1,x_2,x_3)$ the initial cluster of the corresponding cluster algebra, see Figure~\ref{figsnakestraight}. 
 The infinite arc $\zg$  crosses the arcs labeled 1 and 2 of the triangulation alternatingly. The corresponding infinite snake graph, shown on the right of the figure, is a straight snake graph  with alternating tile labels $1,2,1,2,\ldots $ and all interior edges have weight $x_3$. The associated continued fraction is the continued fraction of the golden ratio  $[\,\overline{1}\,]=[1,1,1,\ldots]$. Let $[L_1,L_2,L_3,\ldots]$ be the corresponding continued fraction of rational functions as defined in section \ref{sect cluster}. Thus 
\begin{center}
\begin{figure}
\small 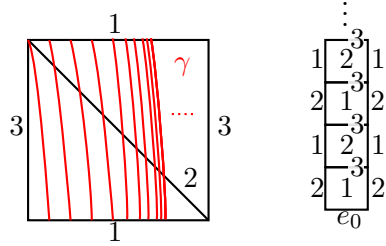
 \caption{A triangulation of the torus with one puncture (in black) and an infinite arc $\zg$ (in red) is shown on the left. The infinite snake graph $\calg$ of $\zg$ is shown on the right. The tiles have labels $1,2,1,2,\ldots$, the south edge $e_0$ and all interior edges have label $3$,   $x_{b_{2i-1}}=x_1, x_{b_{2i}}=x_2,$ for $i=1,2, \ldots$, and $x_{e_0}=x_3$. }
  \label{figsnakestraight}
\end{figure}
\end{center}
 \[L_1= e_0/b_1 = x_3/x_1,\quad L_2= e_1b_1/b_0b_2 = x_3x_1/x_2^2, \quad L_3= e_2b_0b_2/b_1^2b_3 = x_3x_2^2/x_1^3,\]
  \begin{equation}\label{eq51} L_i =\frac{ x_3}{x_2} \left(\frac{x_1}{x_2 }\right)^{i-1}, \textup{ if $i$ is even} \quad ;\quad L_i=   \frac{ x_3}{x_1} \left(\frac{x_2}{x_1 }\right)^{i-1}, \textup{ if $i$ is odd.}
 \end{equation}
 
\begin{lem}
$[L_1,L_2,\ldots]$ converges pointwise whenever $x_1,x_2,x_3\in\mathbb{R}_{>0}$.\end{lem}
\begin{proof} Let $x_1,x_2,x_3\in\mathbb{R}_{>0}$, and denote by $r_i$ the value of $L_i$ at the point $(x_1,x_2,x_3)$. 
From the formulas above, we see that $r_i\in\mathbb{R}_{>0}$.
 By a classical result of Seidel and Stern, see \cite[\S50, Satz 8]{Perron}, the continued fraction $[r_1,r_2,\ldots]$ with $r_i\in \mathbb{R}_{>0}$ converges if and only if the series $\sum r_i$  diverges. Thus, using the divergence test, it suffices to show that  $\lim_{i\to\infty} r_i \ne 0$.
  Suppose to the contrary that the sequence $r_i$ converges to zero. Then also the subsequence $r_{2i}$ with even indices converges to zero. Thus equation (\ref{eq51}) yields
  \[0=\lim_{i\to\infty} r_{2i} = \lim_{i\to\infty} \frac{x_3}{x_2} \left(\frac{x_1}{x_2 }\right)^{2i-1},\] 
which implies that $x_1<x_2$. However, the same argument for the odd indices gives
 \[0=\lim_{i\to\infty} r_{2i+1} = \lim_{i\to\infty} \frac{x_3}{x_1} \left(\frac{x_2}{x_1}\right)^{2i-1},\] 
which implies that $x_1>x_2$, a contradiction. 
\end{proof}

 Our next goal is to compute the value of $[L_1,L_2,\ldots]$.
From equation (\ref{eq51}) we see that  $L_{2i+2}=\cfrac{x^2_1}{x^2_2}\,L_{2i} $ and $L_{2i+1}=\cfrac{x^2_2}{x^2_1}\,L_{2i-1}$.  Therefore Lemma \ref{lem cf} implies 
 \begin{equation}\label{eqinf}[L_3,L_4,\ldots] =  \frac{x^2_2}{x^2_1}\,[L_1,L_2,\ldots].\end{equation}
 
 Let $\za=[L_1,L_2,\ldots]$. Then 
 \[ \za=L_1+ \cfrac{1}{L_2 +\cfrac{1}{[L_3,L_4, \ldots]}} 
 =\cfrac {L_1L_2[L_3,L_4,\ldots] + L_1+[L_3,L_4,\ldots] }{L_2[L_3,L_4,\ldots] +1} .\]
 Using equations (\ref{eq51}) and (\ref{eqinf}), we get
 \[\za =\left. \left( \frac{x_3^2}{x_1^2} \za +\frac{x_3}{x_1}+\frac{x_2^2}{x_1^2}\za \right)\right/ \left( \frac{x_3}{x_1}\za +1\right)\]
 and thus
 \[  {x_3}\za^2 +\za\left(x_1-\frac{x_2^2}{x_1}-\frac{x_3^2}{x_1}\right)
 -{x_3}=0.\]
 Using the short hand   $x_1'=\displaystyle\frac{x_2^2+x_3^2}{x_1}$ we obtain
 \[ x_3\za^2 +\za(x_1-x_1')
 -x_3=0.\]
  
Now the quadratic formula yields
\[\za = \frac{{ x'_1-x_1} \pm\sqrt{\left( x'_1-x_1\right)^2+4x_3^2}}{2x_3}. \]

If we consider $\za$ as a function of $x_1,x_2,x_3\in \mathbb{R}_{>0}$, we have $\za>0$ and thus the sign in front of the radical must be positive.
   
We have proved part (a) of the following result.

\begin{thm}\label{thm inf} Let $\cala$ be 
 the cluster algebra of the once punctured torus with initial cluster $\x=(x_1,x_2,x_3)$. 
 Let $x'_1=(x_2^2+x_3^2)/x_1$ denote the cluster variable obtained from $\x$ by the single mutation $\mu_1$. Let  $u(i)$ denote the cluster variable obtained from $\x$ by the mutation sequence  $\mu_1\mu_2\mu_1\mu_2\mu_1\cdots$ of length $i$, and let $\calg(i)$ be its snake graph. Note that $\calg(i)$ is a straight snake graph with $2i-1$ tiles, with alternating tile labels $1,2,1,2,\ldots 1$ and all interior edges have weight $x_3$. Let  $v(i)$ be the  Laurent polynomial corresponding to the snake graph $\calg(i)\setminus 
\calh_1$ 
obtained from $\calg$ by removing the  edge $e_0$.  
\begin{itemize}
\item [\textup{(a)}]
 The quotient $u(i)/v(i)$   converges as $i$ tends to infinity and its limit is  
\[\za =  \frac{{ x'_1-x_1} +\sqrt{\left({ x'_1-x_1}\right)^2+4x_3^2}}{2x_3}. \]

\item [\textup{(b)}]
The quotient of the cluster variables $u(i)/u(i-1)$ converges as $i$ tends to infinity and its limit is 
\[\zb
= \frac{x_1'+x_1+\sqrt{\left({x_1'-x_1}\right)^2+4x_3^2}}{2x_2} =\frac{x_3\za+x_1}{x_2}.\]

\item [\textup{(c)}]
  Let $z={ \frac{ x'_1-x_1}{x_3} }$. Then $\za$ has the following  periodic continued fraction expansion
\[\za= \left\{\begin{array}{ll}\left[\,\overline{ z}\,\right] , &\textup{if $z  >0$;}\\ \\
\left[ 0, \overline{-z }\,\right] &\textup{if $z < 0$} .\end{array}\right.\]
\item [\textup{(d)}]
 Equivalently
\[\za =\left[ \,\overline{\left(\frac{x_2^2+x_3^2-x_1^2}{x_1x_3}\right)} \,\right]  \ or\  \za =\left[ \,0,\overline{\left(\frac{x_1^2-x_2^2-x_3^2}{x_1x_3}\right)} \,\right] \]

\end{itemize}
 \end{thm}
\begin{proof} Part (a) has already been proved above. 

(b) To study  the quotient of the cluster variables $u(i)/u(i-1)$, we need to consider the snake graph in Figure \ref{figsnakestraight}, except that we interchange the labels $e_0$ and $b_0$. The continued fraction of this snake graph is $[2,\overline{1} ] =\frac{3+\sqrt{5}}{2}.$  The subgraph $\calh_1$ is equal to the first tile and $x(\calh_1)=\frac{x_2^2+x_3^2}{x_1}=x_1'$. Denoting the corresponding Laurent polynomials by $L_i'$, we have
\[L_1'=\frac{x_1'}{x_2} , 
\qquad
L_2'=\frac{e_2 b_1}{b_0b_2}=\frac{x_2}{x_1} =L_1 \frac{x_2}{x_3} ,
\qquad
L_3'=\frac{e_3 b_0b_2}{b_1^2 b_3}=\frac{x_1x_3^2}{x_2^3} =L_2 \frac{x_3}{x_2} 
\]
and for $i\ge 2$, 
\[L_i'=\left\{\begin{array}{ll}L_{i\!-\!1}\frac{x_2}{x_3} &\textup{if $i$ is even;}\\
L_{i\!-\!1}\frac{x_3}{x_2} &\textup{if $i$ is odd.}\end{array}\right.
\]
Let $\zb=[L_1',L_2',\ldots]$ and $\za'=[L_2',L_3',\ldots]$. Then using Lemma~\ref{lem cf}, we have
$\za'=\frac{x_2}{x_3}\za,$ and
\[\zb=L_1'+\frac{1}{\za'}
 = \frac{x_1'+\frac{x_3}{\za}}{x_2} 
= \frac{x_1'+x_1+\sqrt{(x_1'-x_1)^2+4x_3^2}}{2x_2} ,\]
where the last identity holds because $\frac{1}{\za}=\frac{- x'_1+x_1+\sqrt{\left({ x'_1-x_1 }\right)^2+4x_3^2}}{2x_3}$. This shows the first identity in (b) and the second is a direct computation. 

(c) If $z>0 $, then $\left[\,\overline{ z}\,\right]>0$,  and
  \[\left[\,\overline{ z}\,\right]=   z +\frac{1}{\left[\,\overline{ z}\,\right]} . \]
 Therefore
 \[\left[\,\overline{ z}\,\right]^2 - z\left[\,\overline{ z}\,\right]-1=0,\] and  from the quadratic formula, we get
 \[\left[\,\overline{ z}\,\right]=\frac{z +\sqrt{z^2+4}}{2} =\za,\]
 where the sign in front of the radical is + because $\left[\,\overline{ z}\,\right]>0$. 
On the other hand, if $z<0$, then 
the computation above yields 
 \[\left[\,\overline{ -z}\,\right] = \frac{-z +\sqrt{z^2+4}}{2} .\]
  and its multiplicative inverse is 
  \[\left[0,\overline{- z}\,\right]  = \frac{z+\sqrt{z^2+4}}{2}=\za.\]
  
This shows (c), and (d) is obtained by replacing $x_1'$ by $\frac{x_2^2+x_3^2}{x_1}$.
\end{proof}

\begin{cor}  Considering $\za$  as a function of $x_1,x_2,x_3$, we have the following.
\[\begin{array}{crcll}
(a) & \za(rx_1,rx_2,rx_3)&=&\za(x_1,x_2,x_3),\textup{ for all $r\ne 0$}.
\\
(b) & 
\za(1,1,1)&=&\left[\,\overline{1}\,\right]=\frac{1+\sqrt{5}}{2} \textup{   is the golden ratio.}\\
&\za(1,1,2)&=&\left[\,\overline{2}\,\right]=1+\sqrt{2} \textup{   is the silver ratio.}\\
 &\za(1,1,n) &=&[\overline{n}]=\frac{n+\sqrt{n^2+4}}{2},\textup{  is called metallic mean.}  \\
(c)& \za(1,n,1)& =&\left[\overline{n^2}\right].\\
(d) &\za(n,1,1) &=&\left[ 0, \overline{(n^2-2)/n}\right]=\left[0,\overline{n-1,1,n-1}\right].\\
(e) &\za(2n+1,1,2)& =&\left[0,\overline{\frac{2(n^2+n-1)}{2n+1}}\,\right] = [0,\overline{ n,2,n}\,].\\
(f)& \za(1,1,\pm 2i) &= &\pm i, \textup{ where $i=\sqrt{-1}\in \mathbb{C}$.}\\
(g)& \zb(1,1,\pm 2i) &=& -1.
\end{array}\]
\end{cor}

\begin{proof}
  This follows directly from part (d) of Theorem \ref{thm inf}. 
 \end{proof}

\begin{remark}
In a very similar way, we can  compute the function $\za'=[L'_1,L'_2,\cdots] $ starting from the continued fraction $[2,2,\ldots]$. More precisely, let $u'(i)$ be the cluster variable obtained from $\mathbf{x}=(x_1,x_2,x_3)$ by the mutation sequence $(\mu_2\circ\mu_1)^i\circ\mu_3$, and let $\calg'(i)$ be its snake graph. Then $\calg'=\calg[2,2,2,\ldots,2]$. Let  $v(i)$ be the  Laurent polynomial corresponding to the snake graph $\calg'(i)\setminus \calh'_1$ obtained from $\calg'$ by removing the subgraph $\calh'_1$.  Then the quotient $u'(i)/v'(i)$   converges as $i$ tends to infinity and its limit $\za'$  is obtained from $\za$ by substituting $x_3$ by $x_3'=(x_1^2+x_2^2)/x_3$, that is  
\[\za' =  \za|_{x_3\leftarrow x_3'} .\] 
This can be seen either by direct computation, or by the fact that the substitution corresponds to the first mutation $\mu_3$ in the sequence leading to $\za'$, and the rest of the mutation sequence is equal to the mutation sequence leading to $\za$.
\end{remark}

\smallskip
Theorem \ref{thm inf} implies the following analogous result for the cluster algebra of the annulus with one marked point on each boundary component, also given by the Kronecker quiver. 

\begin{cor} Let $\cala $ be the cluster algebra  of the quiver $\xymatrix{1\ar@<2pt>[r]\ar@<-2pt>[r]&2}$ with initial cluster  $\x=(x_1,x_2)$. Let $x'_1=(x_2^2+1)/x_1$ denote the cluster variable obtained from $\x$ by the single mutation $\mu_1$. Let  $u(i)$ denote the cluster variable obtained from $\x$ by the mutation sequence  $\mu_1\mu_2\mu_1\mu_2\mu_1\cdots$ of length $i$, and let $\calg(i)$ be its snake graph.  Let  $v(i)$ be the  Laurent polynomial corresponding to the snake graph $\calg(i)\setminus 
\calh_1$ 
obtained from $\calg$ by removing the  edge $e_0$.  

\begin{itemize}
  \item [\textup{(a)}]
 The quotient $u(i)/v(i)$   converges as $i$ tends to infinity and its limit is  
\[\za =  \frac{{ x'_1-x_1} +\sqrt{\left({ x'_1-x_1}\right)^2+4}}{2}. \]
\item [\textup{(b)}]
The quotient of the cluster variables $u(i)/u(i-1)$ converges as $i$ tends to infinity and its limit is 
\[\zb
= \frac{x_1'+x_1+\sqrt{(x_1'-x_1)^2+4}}{2x_2} .\]

\item [\textup{(c)}]
 The limit $\za$ has the following  periodic continued fraction expansion
\[\za= \left\{\begin{array}{ll}\left[\,\overline{ { { x'_1-x_1}}}\,\right] , &\textup{if ${ { x'_1>x_1}} $;}\\ \\
\left[ 0, \overline{{ { x_1-x'_1}} }\,\right] &\textup{if ${ { x'_1<x_1}}$} .\end{array}\right.\]
\item [\textup{(d)}]
 Equivalently
\[\za =\left[ \,\overline{\left(\frac{x_2^2+1-x_1^2}{x_1}\right)} \,\right]  \ or\  \za =\left[ \,0,\overline{\left(\frac{x_1^2-x_2^2-1}{x_1}\right)} \,\right] \]

\end{itemize}

\end{cor}
\begin{proof}
This follows from Theorem \ref{thm inf} by specializing $x_3=1$. 
\end{proof}

Acknowledgements: We thank  Dmitry Badziahin,  Keith Conrad, Anna Felikson and Andy Hone for  stimulating discussions.

{}

 \end{document}